\documentclass{amsart}
\usepackage{lineno}
\modulolinenumbers[5]
\usepackage{setspace}
\usepackage{hyperref}
\usepackage{amsmath}
\usepackage{amssymb}
\usepackage{amsfonts}
\usepackage{bm}
\usepackage{amsthm}
\usepackage{array,epsfig,fancyhdr}
\usepackage[sectionbib,numbers]{natbib}
\usepackage[utf8]{inputenc}
\usepackage{CJKutf8}
\usepackage{tikz}
\usepackage{multirow}

\usepackage{xr}
\newtheorem{theorem}{Theorem}[section]
\newtheorem{lemma}[theorem]{Lemma}
\newtheorem{corollary}[theorem]{Corollary}

\newtheorem{conjecture}[theorem]{Conjecture}

\theoremstyle{definition}
\newtheorem{definition}[theorem]{Definition}
\newtheorem{example}[theorem]{Example}

\theoremstyle{remark}
\newtheorem{remark}[theorem]{Remark}

\numberwithin{equation}{section}

\DeclareGraphicsExtensions{.pdf,.png}
\usepackage{wrapfig}
\usepackage{lipsum}

\begin{document}

\title[A sequel to the adventure of RGB-tilings]{A sequel to the adventure of RGB-tilings \\ to explore the Four Color Theorem}

\author{Shu-Chung Liu}
\address{Institute of Learning Sciences and Technologies, National Tsing Hua University, Hsinchu, Taiwan}
%\curraddr{Institute of Learning Sciences and Technologies, National Tsing Hua University, Hsinchu, Taiwan}
\email{sc.liu@mx.nthu.edu.tw}

%\author{Author Two}
%\address{Mathematical Research Section, School of Mathematical Sciences,
%Australian National University, Canberra ACT 2601, Australia}

%    General info
\subjclass[2020]{Primary 05C10; 05C15}

\date{\today}


\keywords{Four Color Theorem; Kempe chain; edge-coloring; RGB-tiling; diamond route; canal line; $\Sigma$-adjustment}

        \begin{abstract}
An approach of using RGB-tilings for proving the Four Color Theorem discussed in three previous work is expanded in this paper. A novel methodology and revisions for the methodology in the three aforementioned papers are discussed, and a previously derived result involving three degree-five vertices in a triangular graph is improved. Moreover, a treatment of a novel topic for a graph with six vertices of degree 5 in a dumbbell shape is presented.  
        \end{abstract}

\maketitle

\section{Introduction} \label{sec:Introduction}

This work extends the discussion of using RGB-tilings and red \makebox{(R-)/} green \makebox{(G-)/} blue \makebox{(B-)}tilings for proving the Four Color Theorem in the domain of maximal planar graphs (MPGs) recently proposed in~three papers \cite{Liu2023I, Liu2023II, Liu2023III}. First, the novel methods and ideas introduced in these papers are briefly reviewed in Sections~\ref{sec:Introduction} and~\ref{sec:rotationDualKempe2}. Subsection~\ref{sec:ABetterResult} presents an improvement of a result in ~\cite{Liu2023III}. Finally, Section~\ref{sec:dumbbell} discusses a dumbbell-shaped graph with six vertices of degree 5. All graphs in this paper are simple and connected.

First, some definitions are provided.

An MPG $M$ can by defined with the following two equivalent definitions. 
   \begin{itemize}
   \item If $M$ is planar and if an edge is added linking any two nonadjacent vertices $a$ and $b$ in $M$, then the new graph $M\cup\{ab\}$ is a nonplanar graph.
   \item \ If $M$ is planar with at least two facets, then all facets of $M$ are triangles.
   \end{itemize}
The second definition is often simpler to use in practice. These definitions differ slightly in that the first definition includes the empty graph, the single-vertex graph, and the graph consisting of only one edge; however, the second definition is only applicable if $|M|\ge 3$.   

The Four Color Theorem claims that all planar graphs are 4-colorable. This claim is equivalent to ``all MPG's are 4-colorable'' as demonstrated by Theorem~\ref{RGB1-thm:4colorMPG} in~\cite{Liu2023I}). 
%%author  
%%the last edition is 
%%``The Four-Color theorem can be proved by contrapositive.''
%%%However, there is no unplugged proof so far. 
%%%Is my change OK?
If a unplugged proof (without checking by computer) of the Four Color Theorem were complete, most researchers expect a  contrapositive proof. Consider the following two sets:
   \begin{eqnarray*}
   \mathcal{N}4 &=& \{G \mid \text{$G$ is a non-4-colorable simple connected planar graph.}\} \\
   e\mathcal{MPGN}4 &=& \{EP \in\mathcal{N}4 \mid \text{$EP$ is an MPG with a minimal number of vertices.}\}    
   \end{eqnarray*}
\noindent
The set $e\mathcal{MPGN}4$ is easier to handle than $\mathcal{N}4$. A relation between  $e\mathcal{MPGN}$ and $\mathcal{N}4$ is presented in Section~\ref{RGB1-sec:eMPGN4} in~\cite{Liu2023I}. 

   \begin{quote}
To prove the Four Color Theorem by contrapositive, throughout this paper, we assume that $e\mathcal{MPGN}4$ is nonempty. Many situations with additional restrictions on $EP\in e\mathcal{MPGN}4$ are discussed. For example, \cite{Liu2023III} examines a case in which two or three degree-5 vertices are neighbors in $EP$.  
   \end{quote}
   
\subsection{Vertex-colorings and RGB-edge-colorings} \label{sec:RGB1}

Normally, natural numbers are used to color vertices such that any two adjacent vertices have different colors as indicated by different numbers. A corresponding RGB edge-coloring naturally exists if the graph has a vertex-coloring function $f:V(G)\rightarrow \{1,2,3,4\}$ (4-colorable), regardless of whether $G$ is planar. Besides $f$, the notation $Co[\ldots]$ is a mapping for both vertex-colorings and edge-colorings.

Assume that $K_4$ is a basic graph to be colored with $\{1,2,3,4\},$ as displayed in Figure~\ref{fig:1234EdgeColor}. The leftmost graph $K_4$ has only vertices and edges with no vertex or edge coloring. The two vertex-colorings of $K_4$ of $Co[v_i:i]$ and $Co[v_1:4, v_2:3, v_3:1, v_4:2]$ are shown as the middle and the right graphs in Figure~\ref{fig:1234EdgeColor}, respectively. The corresponding edge-colorings are also displayed. A simple rule can be established as follows: 
    \begin{center}
    \begin{tabular}{rcl}
{1-3 and 2-4 edges} & \multirow{4}{3cm}{shall be painted by} 
& red or R, r;\\
{1-4 and 2-3 edges} &  & green or G, g;\\
{1-2 and 3-4 edges} &  & blue or B, b;\\
{uncertain in RGB}&  & black or bl, sometime gray. 
    \end{tabular}
    \end{center}
   \begin{figure}[h]
   \begin{center}
%   \begin{tabular}{ c }
%   \hspace*{-12pt}   
   \includegraphics{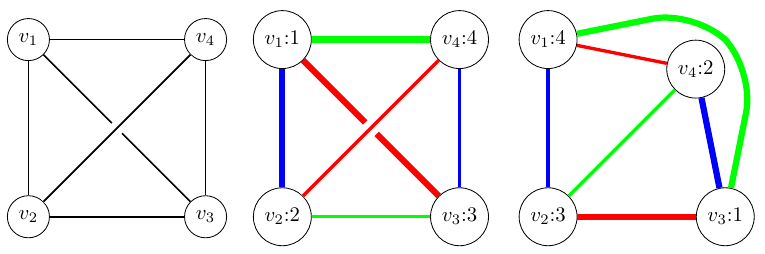}
%   \end{tabular}
   \end{center}
   \caption{Corresponding RGB-edge-colorings} \label{fig:1234EdgeColor}
   \end{figure}  
   
Among these colored edges, those incident to 1 are depicted with a thick line. For a fixed edge-color (e.g., red), the (red) thick edges and thin ones form two subclasses. These two subclasses are not only disconnected for red; if we assume nonred edges are black, then $(V(G),\{\text{black edges}\})$ forms a bipartite graph, with the partite sets made by the vertices of these two subclasses separately.  

Assigning the edges incident to 1 as ``thick'' red, green, or blue is not mandatory; the goal is simply to distinguish the two disconnected subclasses colored for each color (red, green, or blue). This disconnection property is the main idea of Kempe's proof. For any graph with a 4-coloring function, we can immediately obtain an RGB edge coloring. However, the reverse is not always true without additional restrictions.

\subsection{Tilings and 4-colorings for MPG's and semi-MPG's}    \label{sec:T4MPG}

A \emph{semi-MPG}, for example, $M$, is nearly an MPG but has some facets, called \emph{outer facets}, that form the border of the structure as a planar graph. Typically, an outer facet of $M$ is an $n$ sided polygon ($n$-gon) with $n\ge 4$. However, in some rare cases, $3$-gons can be the outer facets of $M$ if they are necessary to define the border of the structure of $M$, as discussed in Sections~\ref{RGB1-sec:RGB4coloring} and~\ref{RGB1-sec:Grandline} in~\cite{Liu2023I}. If the $n$-gon has a single outer facet, $M$ is called an $n$-semi-MPG. Generally, an $(n_1,n_2,\ldots,n_k)$-semi-MPG has $k$ outer facets with sizes $n_1,n_2,\ldots,n_k$. Typically, we forbid any two outer facets from sharing an edge because every edge must belong to one or two triangles. Some exceptions to this rule may occur, such as in Lemma~\ref{RGB1-thm:evenoddRGB}(a') in~\cite{Liu2023I}.

In particular, an MPG or an $n$-semi-MPG is called \emph{one piece} because all of the loops in these structures are free loops that can be shrunk to a single point topologically. However, an $(n_1,n_2,\ldots,n_k)$-semi-MPG is not one piece.  A brief sketch of this situation is as follows:

   \begin{center}
   \begin{tabular}{rcl}
  & MPG & \multirow{2}{2.4cm}{{\huge \} }\quad \emph{One Piece}}\\
\multirow{2}{2.5cm}{\emph{semi-MPG}\quad {\huge \{ }} & $n$-semi-MPG\\  
    & $(n_1, n_2, \ldots n_k)$-semi-MPG   
   \end{tabular}
   \end{center}

The first graph in Figure~\ref{fig:reddiamond} is a $(5,7)$-semi-MPG. The basic elements of an MPG or a semi-MPG are triangles. Two triangles sharing an edge constitute a \emph{diamond}. In Figure~\ref{fig:reddiamond}, the edge shared by the two triangles is colored red. This structure is called a \emph{red tile} or a \emph{red diamond}. The middle graph in Figure~\ref{fig:reddiamond} is a $v_2v_4$-diamond, and the \emph{four surrounding edges} of $v_2v_4$ are $v_1v_2$, $v_2v_3$, $v_3v_4$, and $v_4v_1$.  The third graph is a \emph{red half-tile} or a \emph{red triangle}. Green or blue tiles and half-tiles can also exist. For an \emph{RGB-tiling} on $M$, each edge can only be colored with red, green, or blue, and each edge of a triangle must be colored with a different color. We first consider a red tiling \emph{R-tiling} on $M$ for which each triangle has one red edge, with the remaining two edges being initially black.
   \begin{figure}[h]
   \begin{center}
   \includegraphics[scale=0.7]{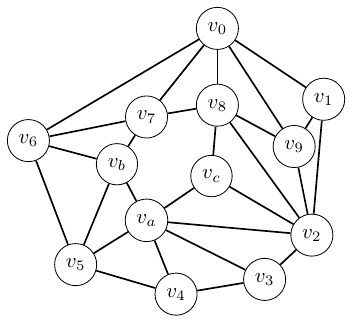}\qquad \quad
   \includegraphics[scale=0.6]{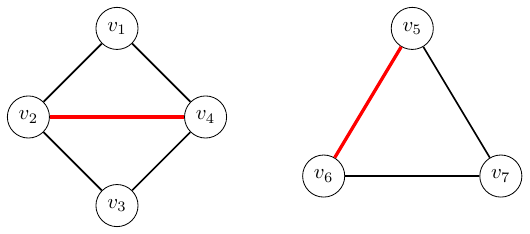}
   \end{center}
   \caption{Red tile $v_2v_4$-diamond and a red half-tile} \label{fig:reddiamond}
   \end{figure}

Let $M$ be an MPG or a semi-MPG. We first try to color exactly one edge of each triangle red; both this process and result are referred to as an \emph{R-tiling} on $M$. 
   \begin{definition} \label{def:Rtiling}
An \emph{R-tiling} is equivalent to a function $T_r:E(M)\rightarrow \{\text{red, black}\}$ such that every triangle of $M$ has one red and two black edges.  An \emph{RGB-tiling}, denoted  by $T_{rgb}:E(M)\rightarrow \{\text{red, green, blue}\},$ is established such that every triangle of $M$ has exactly three colors.
   \end{definition}

Figure~\ref{fig:twoRtilings} presents three R-tilings denoted by I, II, and  III. A 4-coloring function is available for only the background graph for R-tiling II; it does not exist for R-tilings I and III because
of the following reasons:   \begin{figure}[h]
   \begin{center}
   \begin{tabular}{ c c }
   \includegraphics[scale=0.67]{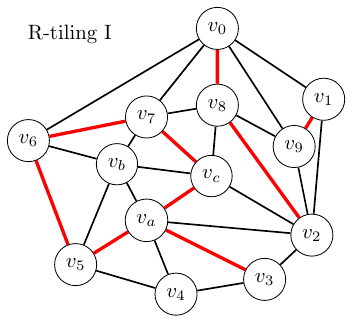}\
   \includegraphics[scale=0.67]{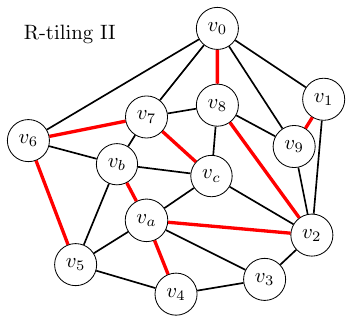}\
   \includegraphics[scale=0.67]{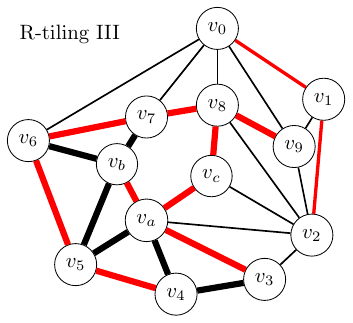}
   \end{tabular}
   \end{center}
   \caption{Three R-tilings; only the middle tiling induces a 4-coloring.}  \label{fig:twoRtilings}
   \end{figure}
   
\noindent
   \begin{center}
   \begin{tabular}{lll}
R-tiling I  & : & There is a red odd-cycle.\\
R-tiling II & : & The R-tiling is \emph{grand} without red odd-cycles. \\
R-tiling III & : & This R-tiling is not \emph{grand}. 
   \end{tabular}
   \end{center}   
Any think red line is assumed to be vertex-colored by $1$ and $3$ alternately, so is any thin red line by $2$ and $4$; thus, a red odd-cycle cannot induce a 4-coloring function. A grand R-tiling is defined as follows:
   \begin{definition}  \label{def:grandGrand}
Let $M$ be an MPG or semi-MPG with an R-tiling $T_r$. Here, two outer facets are allowed to share edges. Because these shared edges belong to no triangles, they can be painted either red or black. Let $M_r$ and $M_{bl}$ denote subgraphs of $M$ comprising all vertices, and only the red and black edges, respectively. The R-tiling $T_r$ is a \emph{grand} if $V(M)$ can be partitioned into two disjoint parts $V_{13}$ and $V_{24}$ such that $M_{bl}$ is a bipartite graph with bipartite vertex sets $V_{13}$ and $V_{24}$ and that no red edges link $V_{13}$ and $V_{24}$. 
   \end{definition} 
   \begin{remark}
The subgraph $M_{bl}$ given by R-tiling III is not bipartite and then not grand because the thick black edges link many pairs of vertices in $V_{13}$. There is another situation that causes $M_{bl}$ not bipartite but $k$-partite with $k\ge 3$. For instance, let us modify R-tiling III by switching the edge-colors as $Co[v_cv_2:r, v_cv_8:bl, v_cv_a:bl]$. Unfortunately, this new R-tiling is still not grand because $M_{bl}$ is tripartite and not bipartite.  
   \end{remark}

   \begin{figure}[h]
   \begin{center}
   \begin{tabular}{ c c }
   \includegraphics[scale=0.9]{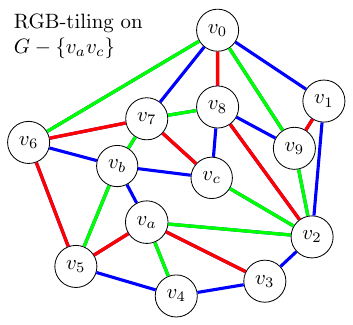}\
   \includegraphics[scale=0.9]{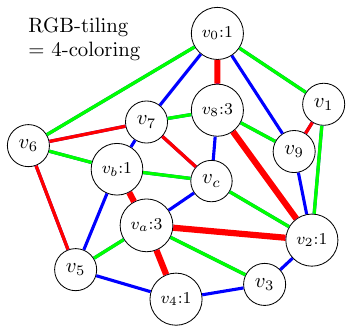}
   \end{tabular}
   \end{center}
   \caption{Two RGB-tilings; the right one induces a 4-coloring.} \label{fig:twoRGBtiling}
   \end{figure}
Let us attempt to convert R-tilings I and II to RGB-tilings. The result is shown in Figure~\ref{fig:twoRGBtiling}. The right graph presents an example of an RGB-tiling derived by filling R-tiling II with appropriate G- and B-tilings; the left graph presents a similar attempt for R-tiling I. However, edge $v_av_c$ must be removed to achieve an RGB-tiling.

These two results are directly demonstrated as follows:
   \begin{theorem}[Theorem for One Piece, Theorem~\ref{RGB1-thm:RtilingOnePiece} in~\cite{Liu2023I}]    \label{thm:RtilingOnePiece}
Every R-tiling on a One Piece (which is either an MPG or an $n$-semi-MPG, $n\ge 3$) is grand.
   \end{theorem}

   \begin{theorem}[The First Fundamental Theorem v1: for R-/RGB-tilings and 4-colorability, Theorem~\ref{RGB1-thm:4RGBtiling} in~\cite{Liu2023I}] \label{thm:4RGBtiling}
Let $M$ be an MPG or an $n$-semi-MPG ($n\ge 4$). Then, the following are equivalent:
   \begin{itemize}
\item[(a)] $M$ is 4-colorable.
\item[(b)] $M$ has an RGB-tiling.
\item[(c)] $M$ has an R-tiling without red odd-cycles.
   \end{itemize}
   \end{theorem}
   
   \begin{corollary}[Corollary~\ref{RGB1-thm:4RGBtiling2} in~\cite{Liu2023I}] \label{thm:4RGBtiling2}
Let $M$ be an MPG or an $n$-semi-MPG ($n\ge 4$). The graph $M$ is non-4-colorable if either no R-tiling on $M$ exists or every R-tiling on $M$ has at least one red odd-cycle.
   \end{corollary}  

\subsection{RGB-tilings of Types A, B, C and D on $EP-\{e\}$}

First, we must state a key property:
   \begin{theorem}[Theorem~\ref{RGB1-thm:eMPG4}(b) in~\cite{Liu2023I}] \label{thm:eMPG4}
If  $EP\in e\mathcal{MPGN}4$, any proper subgraph of $EP$ is 4-colorable, and any MPG graph $G$ with $|G|<|EP|$  is 4-colorable.
   \end{theorem}
Given $EP\in e\mathcal{MPGN}4$ and an edge $e=ab\in E(EP)$, $Q:=EP-\{e\}$ is a proper subgraph of $EP$. This 4-semi-MPG $Q$ is 4-colorable. The possible RGB-tilings on $Q$ are interesting. Let us investigate four types of RGB-tilings, namely A, B, C, and D, on $EP-\{e\}$ as follows:
   \begin{figure}[h]
   \begin{center}
   \begin{tabular}{ c c c c}
   \includegraphics[scale=0.85]{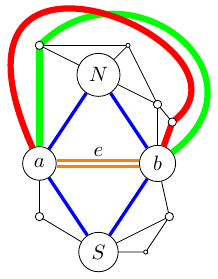} &
   \includegraphics[scale=0.85]
{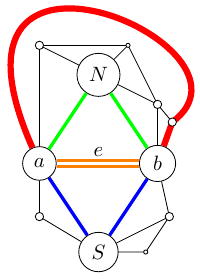} &
   \includegraphics[scale=0.85]{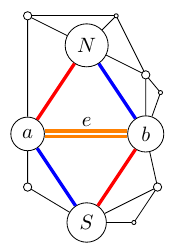} &
   \includegraphics[scale=0.85]{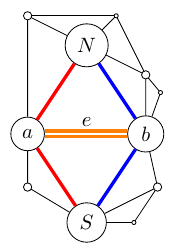}
   \end{tabular}
   \end{center}
   \caption{RGB-tilings of Types A, B, C, D on $EP-\{e\}$} \label{fig:impossible}
   \end{figure}

Because of the symmetry of the colors R, G, and B, every representative type in Figure~\ref{fig:impossible} has 6 \emph{synonyms} obtained by any permutation of R, G, and B on the entire  $EP-\{e\}$. Second, the yellow double-line (edge $e$) is currently an \emph{abandoned edge}. The red curves in Types A and B and the green curve in Type A only are called \emph{Kempe chains} with regard to the $e$-diamond. There is no Kempe chain for Types C and D. Indeed, we can color $e$ by green for Types C and D without causing any problems; thus $e$ is not abandoned for 4-colorability. We call the $e$-diamonds of Types C$^\ast$ and D$^\ast$ respectively after coloring $e$ by green, as well as their 6 synonyms.

For these four types, the fundamental setting is that the four edges surrounding $e$ have either one color or two colors in pairs for the following three reasons, particularly the part (b):
   \begin{lemma}[Lemma~\ref{RGB1-thm:evenoddRGB} in~\cite{Liu2023I}]  \label{thm:evenoddRGB}
Let $M$ be an $n$-/$(n_1,n_2,\ldots,n_k)$-semi-MPG for $n, n_i\ge 3$.
  \begin{itemize}
\item[(a)] If $M$ has an R-tiling, then the number of black edges along $\Omega(M)$ (the set of edges along all outer facets of $M$) must be even. The remaining edges along $\Omega(M)$ are red, and each is associated with a red half-tile.

\item[(a')] Additionally, if two outer facets can share edges in $M$, these shared edges are counted with multiplicity 2 in the index $||\Omega(M)||:=\sum_{i=1}^k n_i$. A single edge that is shared by two outer facets is associated with no triangles; thus, their colors can be freely chosen as red and black. The result of item {\rm (a)} is still valid in this case.

\item[(b)] If $M$ has an RGB-tiling, the three groups of edges along $\Omega(M)$ sorted as red, green, and blue must either all have even or odd cardinality. In particular, if $|\Omega(M)|$ is even, they are all even; if $|\Omega(M)|$ is odd, they are all odd.   
   \end{itemize}
   \end{lemma}

   \begin{lemma}
Let $M$ be a 4-colorable MPG or semi-MPG. The RGB-tiling induced a 4-coloring function of $M$ has all diamonds to be Types C$^\ast$ and D$^\ast$.  
   \end{lemma}
   
The following is a new theorem that improves on the conclusion of the discussion in Sections~\ref{RGB2-sec:ediamond}, \ref{RGB2-sec:NecSufConds}, and~\ref{RGB2-sec:TypeAisSyndrom} in~\cite{Liu2023II}. Here we define the unique cardinality  $\omega:=|EP|$ for any $EP\in e\mathcal{MPGN}4$.

   \begin{theorem}[Fundamental Theorem: necessary  and sufficient conditions]  \label{thm:RGBNesSuf}
Let $M$ be an MPG with $|M|\le \omega$ and given any $e=ab\in E(M)$. Also let $Q:=M-\{e\}$. 
   \begin{itemize}
\item[(a)] $M\notin  e\mathcal{MPGN}4$ if and only if we can find an RGB-tiling of Type C or D  on the 4-semi-MPG $Q$.
\item[(b)] $M\in  e\mathcal{MPGN}4$ if and only if every RGB-tiling on the 4-semi-MPG $Q$ is always Type A or B. 
   \end{itemize}
   \end{theorem}

	\begin{remark} \label{re:Mleomega}
Why this theorem sets $|M|\le \omega$ but not focuses only on $|M|= \omega$? Given a particular $EP \in  e\mathcal{MPGN}4$ with some sort of requirements, we might modify $EP$ to be a smaller MPG $M$. If we can show Theorem~\ref{thm:RGBNesSuf}(b) for $M$ with help from all kinds of RGB-tilings on $EP-\{e\}$, then we reach a contradiction for $M \in  e\mathcal{MPGN}4$ and $|M|<\omega$.
	\end{remark}

	\begin{remark} \label{re:abmerged}	
A single $e$-diamond, say $e=ab$, is usually defined for $EP\in e\mathcal{MPGN}$ together with an RGB-tiling $T$ on $EP-\{e\}$. This tiling for $e$-diamond may be either Type A or B, and this $e$ can be any edge in $E(EP)$. If we merge $a$ and $b$ as one vertex, the new $T'$ induces a usable 4-coloring function on this smaller MPG $EP_{a=b}$. We can extend the definition of a single $e$-diamond for an R-tiling $T_r$ on $EP$. 
On a non-4-colorable MPG, any $T_r$ has at least a red odd-cycle; On $EP$, there is a $T_r$ with exactly one red odd-cycle; On a 4-colorable MPG, 
there is a $T_r$ without red odd-cycles; On $EP_{a=b}$, every $T_r$ without red odd-cycles must have at least a red even-cycle passing through the particular vertex $a=b$. The collection $\{EP_{a=b}\}_{ab\in A}$ for a particular $A\subseteq E(EP)$ is an interesting set to study. 
	\end{remark}

\subsection{Synonym, equivalence, and congruence for $
\mathcal{RGBT}(M)$}
\label{sec:SynEquivCong}

Given a fixed $EP\in e\mathcal{MPGN}4$, let the topic of discussion $T\!D:=(\{u_1, u_2,\ldots, u_k\};\text{requirements})$ comprise some selected vertices from $V(EP)$ where \emph{{requirements}} represent the condition of vertices $u_1, u_2,\ldots, u_k$. Here, $T\!D$ also represents the induced subgraph of the vertices of $T\!D$. Typically, we require this subgraph $T\!D$ to be connected and solid, where solid indicates that it has no holes. The \emph{border} $\Omega$ (sometimes $\Phi$) is a subgraph induced by the vertex set comprising all neighbors surrounding $T\!D$. Clearly, $\Omega$ is a cycle, unless $T\!D$ is unusual. Let  $\Sigma:=\Omega \cup T\!D$ and $\Sigma':=EP-T\!D$. Clearly, $\Sigma\cap\Sigma'=\Omega$. They are called the \emph{interior} and \emph{exterior} of $(EP; \Omega)$, and sometimes denoted by $\Sigma(EP; \Omega)$ and $\Sigma'(EP; \Omega)$ respectively.

If $\Sigma:=e$-diamond as in the last section, this is a special case with the vertex set of $T\!D$ empty. We focus on the $e$-diamond with $\Omega$ comprising the four surrounding edges of $e$.   The second basic case is $T\!D:=(\{v\};\deg(v)=5)$ and is detailed in Section~\ref{RGB3-sec:Gcanallines} in~\cite{Liu2023III}. The most interesting $T\!D$ is $(\{a,b: \text{adjacent}\};\deg(a,b)=5)$ and $(\{a,b,c: \text{in a triangle}\};\deg(a,b,c)=5)$.

The following discussion concerns the topics of ``the same in some sense'' or ``the difference in certain levels'' and involves three general definitions. Let $M$ be an MPG or a semi-MPG, and let $\mathcal{RGBT}(M)$ ($\mathcal{RT}(M)$) be the set of all RGB-tilings (R-tilings) on $M$.

   \begin{itemize}
\item Synonym: Any $T \in\mathcal{RGBT}(M)$ has six \emph{synonyms}, including itself, obtained by interchanging red, green, and blue over the entire graph $M$.  The synonym relationship is a basic idea and trivial. Let $\langle T \rangle$ denote the set of six synonyms of $T$, and any two synonyms $T$ and $T'$ are equated by $T \overset{\text{\tiny syn}}{=} T'$.  The angle brackets may be omitted in the notation. 
\item Equivalence: First, we shall accept each set of six synonyms as one element. Let us use Type A as an example. The most crucial part of Type A is the pair $(K_r|_a^b,K_g|_a^b)$ of Kempe chains in $\Sigma'$. This major structure  in $\Sigma'$ is called the \emph{skeleton} of Type A. The \emph{skeleton} of Type B has only one Kempe chain, namely $K_r|_a^b$. A \emph{skeleton} comprises red, green, or blue connections among some vertices along $\Omega$; for example, $K_r|_a^b$ makes a red connection between $a$ and $b$ in $\Sigma'$ for Types A and B. This connection is  mandatory because $EP\in e\mathcal{MPGN}4$; however, different $T, T' \in  \mathcal{RGBT}(Q)$ might have different $K_r|_a^b$ and $K'_r|_a^b$.  If $T$ and $T'$ are both Type A (or Type B), they shall share the same \emph{skeleton} and are therefore \emph{equivalent}, denoted by $T \equiv_{\Omega} T'$ under $\Omega$. That is, they have the same sketch for $(K_r,K_g)$ and $T|_\Omega = T'|_\Omega$ for edge-colors. All elements of $\langle T \rangle$ comprise an equivalence class denoted by $[T]$. 
\item Congruence: This new relation requires accepting $\langle T \rangle$ or $[T]$ as one element. However, we shall start with two RGB-tilings $T_a$ and $T_b$ as representatives of $[T_a]$ and $[T_b]$ respectively. We say $T_a$ and $T_b$ are \emph{congruent}, denoted by $T_a  \cong T_a$, if $T_b$ can be obtained from $T_a$ by performing a sequence of vertex-color-switching (VCS) and edge-color-switching (ECS) operations. Furthermore, we can say $[T_a]  \cong [T_a]$.  Kempe's classic proof uses VCS; here, ECS is used along a (generalized) canal line or ring, as detailed in Sections~\ref{RGB3-sec:Gcanallines} and~\ref{RGB3-sec:rotationDualKempe2} in~\cite{Liu2023III}. Many examples are also presented in the remainder of this paper.  
   \end{itemize}

\section{Kempe chains around two adjacent vertices of degree 5 in $EP$} \label{sec:rotationDualKempe2}

One interesting $EP\in e\mathcal{MPGN}4$ is that with two adjacent vertices, say $a$ and $b$, of degree 5. ECS can be performed on or along generalized canal rings around $a$ and $b$ to obtain many Kempe chains (skeletons) with different statuses.

Suppose that the topic of discussion $T\!D:=(\{a,b: \text{adjacent}\};\deg(a,b)=5)$, and $e:=ab$. Surrounding $T\!D$ or $e$ is the \emph{border} $\Phi$, which is a cycle with $\Phi:=v_1$-$v_2$-$c$-$v_4$-$v_5$-$d$-$v_1$ in this case. Because $\deg(a,b)=5$, let ``$55$''  denote this particular $EP$. 

The two graphs in Figure~\ref{fig:KempeCRo2deg5basic} are the initial RGB-tilings of $EP$ with a Type-A $e$-diamond under equivalence. The subscripts $\alpha$ and $\beta$ are simply labels to distinguish them. Clearly, $[T_\alpha^{55}]\not\equiv [T_\beta^{55}]$. The equivalence class $[T_\alpha^{55}]$ of the RGB-tilings is clearly symmetric with respect to both the vertical and horizontal lines; so is the class $[T_\beta^{55}]$.

Because we assume this $55$-$EP$ exists and by Theorems~\ref{thm:4RGBtiling} and~\ref{thm:eMPG4}, at least one of $[T_\alpha^{55}]$ and $[T_\beta^{55}]$ exists. However, we will demonstrate that $[T_\alpha^{55}]\cong [T_\beta^{55}]$; that is, through a sequence of ECS operations, one can be obtained from the other. This congruence identity shows that $[T_\alpha^{55}]$ and $[T_\beta^{55}]$ coexist. 

   \begin{figure}[h]
   \begin{center}
   \includegraphics[scale=0.9]{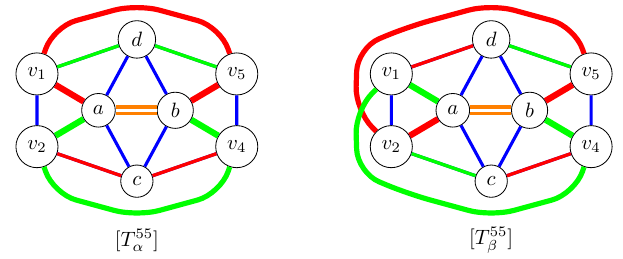}
   \end{center}
   \caption{Two initial RGB-tilings of $EP$ with $T\!D =55$} \label{fig:KempeCRo2deg5basic} 
   \end{figure}

\subsection{Let $e$-diamonds and Kempe chains roll} \label{sec:RnR}

Let  $\Sigma:=\Phi \cup T\!D$ and $\Sigma':=EP-T\!D$. Clearly, $\Sigma\cap\Sigma'=\Phi$.
Starting with the initial status $S_0:=[T_\alpha^{55}]$, we perform ECS 10 times consecutively in accordance with the red and green dashed lines in Figure~\ref{fig:KempeCRo2deg5}. The figure shows the rotation of many $e$-diamonds or Kempe chains around the vertices $a$ and $b$ or around $\Phi$, respectively.   
   \begin{figure}[h]
   \begin{center}
   \includegraphics[scale=0.82]{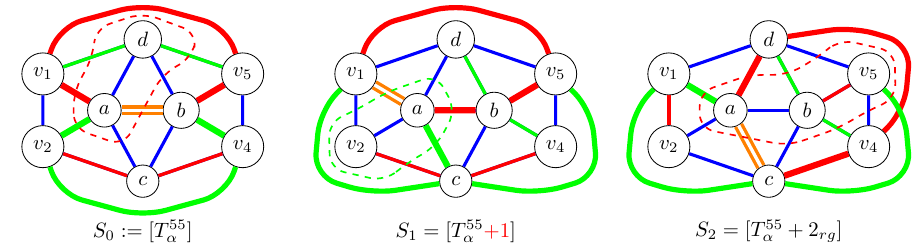}
   \includegraphics[scale=0.82]{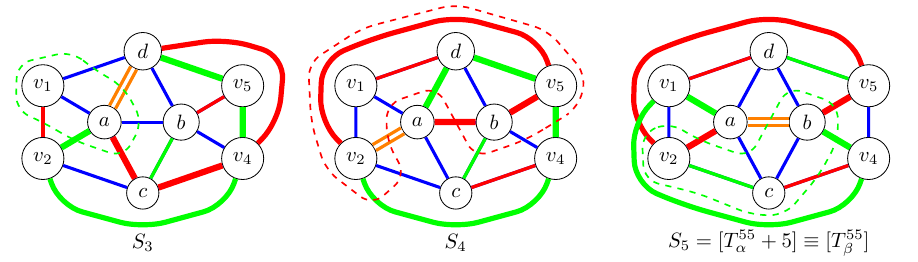}
   \includegraphics[scale=0.82]{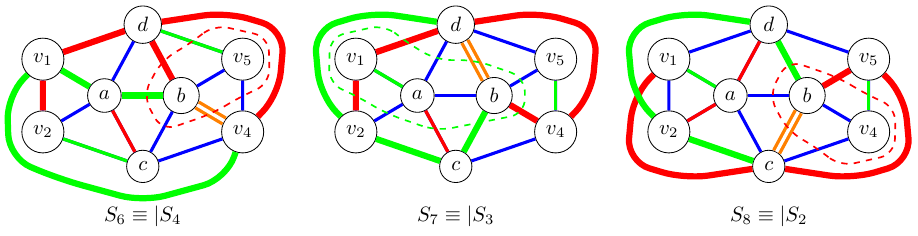}
   \includegraphics[scale=0.82]{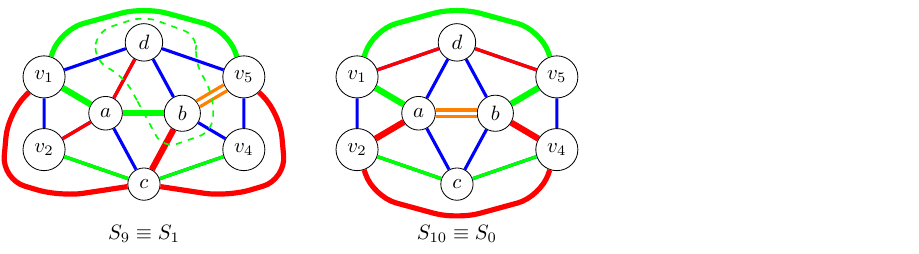}
   \end{center}
   \caption{Rock-n-roll around $(\{a,b\};\deg(a,b)=5)$} \label{fig:KempeCRo2deg5} 
   \end{figure}e

Each red/green dashed line is a generalized canal ring, denoted by $rGCL$/$gGCL$. Normal red (green) dashed lines lie along a canal line (CL) that has two red (green) canal banks on its two sides, and this dashed line alternately crosses the green and blue (red and blue) edges.  These lines are generalized because inside $\Sigma$, this red (green) dashed line might cross one or more yellow double-lines and red (green) edges. A discussion of generalized and normal canal lines or rings is presented in Section~\ref{RGB1-sec:Grandline} in~\cite{Liu2023I} and Section~\ref{RGB3-sec:Gcanallines} in~\cite{Liu2023III}. ECS can simply be performed along a red generalized canal ring $rGCL$ as follows:
   \begin{itemize}
\item Switch the colors of the green and blue edges for the normal part of each red canal line  
\item Switch the colors of the red and yellow double-line edges for the generalized'' segment
   \end{itemize} 
This rule can be easily modified to obtain the process for a green or blue generalized canal ring, as demonstrated by the 10 ECS processes presented in Figure~\ref{fig:KempeCRo2deg5}.

   \begin{remark}
Some details in Figure~\ref{fig:KempeCRo2deg5} should be further explained. By symmetry the graph $S_0$ includes two versions of $rGCL$'s that turns around vertex $a$ or $b$, respectively. Hence, both $K_g|_c^{v_1}$ and $K_g|_c^{v_5}$ are in $S_1$. Moreover, although the graph for $S_1$ seems to have $\deg(c)=6$, it is also possible that $\deg(c)=5$ if $K_g|_c^{v_1}$ and $K_g|_c^{v_5}$ share one green edge incident to $c$. In $S_2$, the unique red Kempe chain can be either $K_r|^d_{v_4}$ or $K_r|^d_{c}$. In $S_4$, the unique green Kempe chain can be $K_g|_{v_2}^{v_4}$, $K_g|_{v_2}^{v_5}$ or  $K_g|_{v_2}^{d}$. However, the choose of these five Kempe chains does not affect our argument in Figure~\ref{fig:KempeCRo2deg5}. Additional remarks are provided in Section~\ref{RGB3-sec:rotationDualKempe2} in~\cite{Liu2023III}.
   \end{remark}

According to Figure~\ref{fig:KempeCRo2deg5}, we have 
  \begin{equation} \label{eq:T55S09}
[S_0=T_\alpha^{55}]_\Phi \cong [S_1]_\Phi \cong \cdots \cong [S_5=T_\beta^{55}]_\Phi \cong \cdots \cong [S_{10}]_\Phi\equiv [T_\alpha^{55}]_\Phi.
  \end{equation} 
We have $[S_{10}]_\Phi\equiv [T_\alpha^{55}]_\Phi$ because they share the same skeleton; however, they may represent two different  RGB-tilings on $EP-\{ab\}$. This difference does not affect the argument. 

Equation~\ref{eq:T55S09} is critical because it indicates that for this type of $EP\in e\mathcal{MPGN}4$ with two adjacent vertices of degree 5, we must only address any one from $S_0$ to $S_5$; $S_6$ to $S_{10}$ are \emph{symmetric} RGB-tilings.  

\subsection{Symmetric RGB-tilings and $ATLAS$ for $T\!D:=(\{a,b\};\deg(a,b)=5)$}  \label{sec:ATLAS}

The notations $|S_\ast$ and $\underline{S_\ast}$ denote the reflection images of $S_\ast$ with respect to the vertical and horizontal lines, respectively. Moreover, $|\underline{S_\ast}$ indicates two reflections. 
Let $[S_\ast]_\text{sym}$ be the equivalence class of these four reflection (\emph{symmetric}) images of $S_\ast$. Typically, $[S_\ast]_\text{sym}$ has four different elements, but both $[T_\alpha^{55}]_\text{sym}$ and $[T_\beta^{55}]_\text{sym}$ have only one element. On the basis of this fact (only one element) and Equation~\ref{eq:T55S09}, we derive the next lemma. This lemma includes a new member $[S_x]$ (shown in Table~\ref{tb:summary}) that is described in Remark~\ref{RGB3-re:Sx0} in~\cite{Liu2023III}. The reader can immediately tell how special $[S_x]$ is.
   \begin{lemma}[Lemma~\ref{RGB3-thm:CongruentEachOther} in~\cite{Liu2023III}]   \label{thm:CongruentEachOther}
All elements in $\{[S_0], [S_1], \ldots, [S_9], [S_x]\}_\text{sym}$ are congruent to each other. The subscript $\text{sym}$ indicates that the set comprises all symmetric images with respect to the vertical and horizontal lines. 
   \end{lemma}
This lemma also suggests that for $EP\in e\mathcal{MPGN}4$ with two adjacent vertices of degree 5, only one of $S_0$ to $S_5$ must be analyzed, in addition to $S_x$. However, this lemma is not complete, because we have not yet considered the last special RGB-tiling $X_2$ in $EP$; therefore, we now study $ATLAS$ which is the collection consisting of all possible $Co(\Sigma)$ allowed using yellow double-lines inside $\Sigma$.

Let Y be the abbreviation of ``yellow double-line''. An edge-coloring function $T :E(M)\rightarrow \{\text{r, g, b, Y}\}$ is called an \emph{eRGB-tiling} on an MPG or semi-MPG $M$ if 
	\begin{itemize}
\item no two Y's in a triangle;
\item the function $T$ restricted to subgraph $M-T^{-1}(Y)$, a semi-MPG allowed two outer facets sharing edges, is an RGB-tiling;
\item each $e$-diamond (associated with a yellow double-line) is either Type A or B.
	\end{itemize}
Now the set $e\mathcal{RGBT}(M)$ is defined naturally.

By examining all possible normal edge-colors without Y along the border $\Phi$, say $Co(\Phi)$, we can obtain $ATLAS$ consisting of 15 eRGB- or RGB-tilings for $T\!D:=(\{a,b\};\deg(a,b)=5)$, where 6 of them are given in Figure~\ref{fig:KempeCRo2deg5} under $[\ast]_\text{sym}$. The background of $ATLAS$ is $EP$ with this $T\!D$, and by induction we can assume that $EP-T\!D$ is 4-colorable and then many different RGB-tilings on $\Sigma'$; however, $Co(\Phi)$ is critical. For convenience, in the remainder of this subsection, {\bf  $[\ast]$ is used as an abbreviation for $[\ast]_\text{sym}$}.

Let the array $(\#r,\#g,\#b)$ denote the numbers of red, green, and blue edges along $Co(\Phi)$; because of Lemma~\ref{thm:evenoddRGB}, these can only be $(0,0,6)$, $(0,2,4),$ and $(2,2,2)$ up to synonyms. The cases of $(0,0,6)$ and $(0,2,4)$ are easy, and their solutions are systematically demonstrated in Table~\ref{tb:summary}.
\begin{table}[h]
\centering
\begin{tabular}{ |c|c|c|c| } 
 \hline
\includegraphics[scale=0.85]{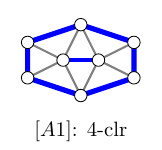} 
&
\includegraphics[scale=0.85] {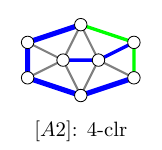}
& 
\includegraphics[scale=0.85] {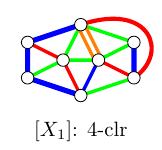} 
& 
\includegraphics[scale=0.85] {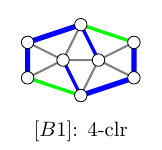}
\\
 \hline 
\includegraphics[scale=0.85] {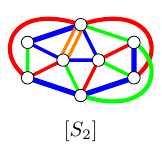}
&
\includegraphics[scale=0.85] {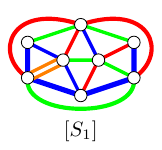}
& 
\includegraphics[scale=0.85] {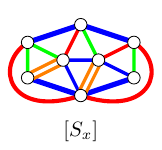} 
& { }
\\ 
 \hline
\end{tabular}\vspace{3mm}
\caption{$Co(\Sigma)$ and skeletons in $\Sigma'$  for $(0,0,6)$ and $(0,2,4)$} \label{tb:summary}
\end{table}
This table reveals that considering skeletons is unnecessary for cases [A$\ast$], [B$\ast$], and [C$\ast$] because they are 4-colorable if we provided an RGB-tiling in $\mathcal{RGBT}(\Sigma')$ with same $Co(\Phi)$. In other words, if we deal with $EP$ as the back ground MPG, [A$\ast$], [B$\ast$], and [C$\ast$] can only match with some eRGB-tilings in $e\mathcal{RGBT}(\Sigma')$ who have at least an $e$-diamond inside $\Sigma'$. However, we draw $K_r|_b^d$, together with $K_b|_b^d$, for case $[X_1]$, which is novel with a Type-A $e$-diamond inside $\Sigma$ and is therefore marked with an X. We argue that $[X_1]$ is 4-colorable because of following lemma.
   \begin{lemma}[Rewrite Lemma~\ref{RGB3-thm:Coalpha} in~\cite{Liu2023III}] \label{thm:CoalphaNew}
Let $EP\in e\mathcal{MPGN}$ and $e=\alpha\beta\in E(EP)$. This $e$-diamond has its surround four vertices $\alpha,N,\beta,S$ and $\deg(N)=5$. One of the following items is provided:  
	\begin{enumerate}
\item[(a)] There is an RGB-tiling $T$ on $EP-\{e\}$ and this $e$-diamond is either Type A or B. The five neighbors of $N$ forms $K_r|_\alpha^\beta \{\alpha\beta: yellow\}$, i.e., a 5-cycle.
\item[(b)]	There is an R-tiling $T_r$ on $EP$ and \underline{every red odd-cycle} passing through $e$. The five neighbors of $N$ forms a red 5-cycle.     
	\end{enumerate}
Then we have $\deg(\alpha,\beta)\ge 6$.  
   \end{lemma}
\noindent Applying this lemma to $[X_1]$ reveals that $K_b|_b^d\cup \{bd\}$ of length 5 surrounds $a$, where $bd$ is an yellow double-line and the it should be $\deg(b,d)\ge 6$. However, $\deg(b)=5$  is a contradiction.

Table~\ref{tb:summary2} lists all possible cases for $(\#r,\#g,\#b)=(2,2,2)$ up to synonyms. 
\begin{table}[h]
\centering
\begin{tabular}{ |c|c|c|c| } 
 \hline
\includegraphics[scale=0.85]{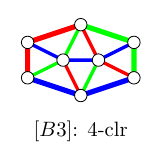} 
&
\includegraphics[scale=0.85] {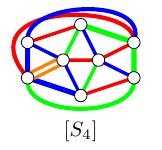}
& 
\includegraphics[scale=0.85] {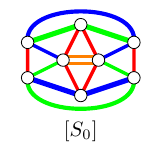} 
& 
\includegraphics[scale=0.85] {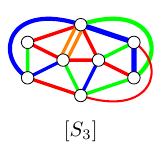}
\\
 \hline 
\includegraphics[scale=0.85] {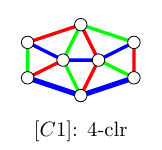}
&
\includegraphics[scale=0.85] {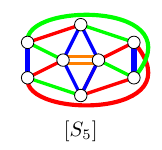}
& 
\includegraphics[scale=0.85] {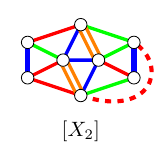} 
&
\includegraphics[scale=0.85] {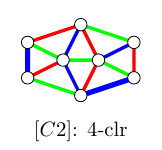}
\\ 
 \hline
\end{tabular}\vspace{3mm}
\caption{$Co(\Sigma)$ and skeletons in $\Sigma'$ for $(2,2,2)$} \label{tb:summary2}
\end{table}
The eRGB-tiling $[X_2]$ is interesting because it has two yellow double-lines. Again $[X_2]\cong [S_0]$. We conclude that 
        \begin{theorem}[Rewritten Theorem~\ref{RGB3-thm:atlas} in~\cite{Liu2023III}] \label{thm:atlasNew}
All non-4-colorable eRGB-tilings in $ATLAS$ (there are eight) are congruent to each other under $[\ast]_\text{sym}$. To study $(EP;55)$, we can first choose any one of these eight.
    \end{theorem} 

The situation involving two yellow double-lines in $\Sigma$ is revisited in several instances in this paper. The pattern $[X_2]$ is discussed in detail in Subsection~\ref{RGB3-sec:ATLAS} and Remarks~\ref{RGB3-re:X2}, \ref {RGB3-re:X2Sx}, and~\ref{RGB3-re:S3} in~\cite{Liu2023III}.   

        \begin{remark}
The theory of ECS on generalized canal rings is the same for all cases with one or more yellow double-lines. The part of the RGB-tiling in $\Sigma'$ (an $n$-semi-MPG) induces a 4-coloring function. Because the segment drawn as a dashed red, green, or blue canal ring in $\Sigma'$ is normal, ECS transforms it into a new RGB-tiling that induces a new 4-coloring function. This new tiling inside $\Sigma$ (inside $\Omega$ or $\Phi$), including new positions for the one or more yellow double-lines, reveal that $EP$ is non-4-colorable. That is, replacing yellow double-lines with red, green, or blue lines would immediately create at least one corresponding odd-cycle. If this odd-cycle is simply a triangle and a facet, it is trivial; we thus focus on the nontrivial Kempe chains.      
        \end{remark}

	\begin{remark}  \label{re:BigPocture}
Let us explain the big picture of the methodology:
	\begin{eqnarray}
&& \text{We assume $\mathcal{N}4$ nonempty. $ \quad \Leftrightarrow \quad  e\mathcal{MPGN}4 \neq \emptyset$} \nonumber\\
&\rightarrow & \text{We focus on an $EP\in  e\mathcal{MPGN}4$ comprising a particular $T\!D$} \nonumber \\
&& \text{(e.g., $55$, $5^3$, $5^4$) as well as the corresponding $\Omega$, $\Sigma$, $\Sigma'$.} \nonumber  \\ 
&\rightarrow & \text{For every fixed $e$ inside $\Sigma$, we draw all possible initial RGB-tilings} \nonumber \\
&& \text{ on $EP-\{e\}$ (e.g., $[T_\alpha^{55}]$ and $[T_\beta^{55}]$). At least one of them exists.} \label{eq:AtLeastOne}\\
&\rightarrow & \text{We should work hard to investigate this $T\!D$.} \nonumber
	\end{eqnarray}	
As a good example associated with~(\ref{eq:AtLeastOne}), we use Figure~\ref{fig:KempeCRo2deg5} to demonstrate that $[T_\alpha^{55}]\cong [T_\beta^{55}]$.  The author is working on a general method to verify a conjecture:
	\begin{conjecture}  \label{conj:IniCong}
For every fixed $e\in E(EP)$, all possible initial RGB-tilings on $EP-\{e\}$ are congruent to each other. 
	\end{conjecture}
	 	\end{remark}

There are two key results involving 3- and 4-cycles in any $EP$ are worth reviewing as follows: 

   \begin{theorem}[Lemma~\ref{RGB1-thm:nontrivial3} in~\cite{Liu2023I} and Theorem~\ref{RGB2-thm:trivial4cycle} in~\cite{Liu2023II}]  \label{thm:nontrivial34}
Let $EP\in e\mathcal{MPGN}$.
        \begin{enumerate}
\item Every 3-cycle in $EP$ forms a 3-facet.
\item If $\Omega:=a$-$b$-$c$-$d$-$a$ is a 4-cycle in $EP$, then either $ac$ or $bd$ is an edge of $EP$, and $a,b,c,d$ induce a single diamond in $EP$.
        \end{enumerate}
   \end{theorem}

\section{Brief revision of Theorem~\ref{RGB3-thm:4deg5inDiamond} and Lemma~\ref{RGB3-thm:5inTri565656} in~\cite{Liu2023III}} \label{sec:AReview}
This revision is an opportunity to practice the proposed methodology. Assume that three vertices of degree 5 in $EP$ form a triangle. Thus, $T\!D:=(\{a,b,c: \text{in a triangle}\};\deg(a,b,c)=5)$ (or ``$5^3$'' for short) has three vertices of degree 5, as shown in Figure~\ref{fig:3deg5a} and $\Omega:=d$-$v_1$-$v_2$-$\ldots$-$v_5$-$d$.

According to Theorem~\ref{thm:atlasNew}, among the eight possible eRGB-tilings on $55$ which are congruent to each other, only one of these tilings must be analyzed; we can thus select the first graph $[T_\alpha^{55}]$ in Figure~\ref{fig:KempeCRo2deg5basic} and in Figure~\ref{fig:KempeCRo2deg5}. For this $[T_\alpha^{55}]$ with $\deg(c)=5$, the two graphs $[T_1^{5^3}]$ and $[T_2^{5^3}]$ in Figure~\ref{fig:3deg5a} are determined by whether $cv_3$ is blue or green. Given that $cv_3$ is blue, a 5-cycle $K_g\cup \{ab\}$ appears. According to Lemma~\ref{thm:CoalphaNew}, $[T_1^{5^3}]$ is impossible. 
   \begin{figure}[h]
   \begin{center}
   \includegraphics[scale=0.87]{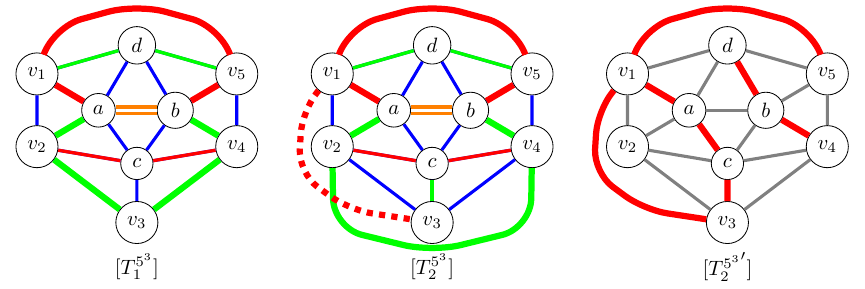}
   \end{center}
   \caption{Only $[T_2^{5^3}]$ is is a possible tiling for $EP$; $[T_1^{5^3}]$ is not} \label{fig:3deg5a} 
   \end{figure} 
Thus, $[T_2^{5^3}]$ is the only possible eRGB-tiling for this $(EP;T\!D)$. In addition, the red dashed path can be shown to exist by simply arranging a new red tiling inside $\Sigma$ and treating green and blue as black. This new red tiling on $EP$ is the third graph in this figure $[{T_2^{5^3}}']$ and must have at least one red odd-cycle that crosses $\Sigma$. Hence, the only possible red path is $K_r|_{v_1}^{v_3}$. Because both $v_1v_2$ and $v_2v_3$ are blue, $K_r|_{v_1}^{v_3}$ is guaranteed to have even length. 

This method for determining the existence of a new $K_r|_{v_1}^{v_3}$ described in the last paragraph is called \emph{$\Sigma$-adjustment}. In short, the method rearranges any valid tiling inside $\Sigma$ with the coloring $Co(\Omega)$ fixed. This reveals more details about the skeleton or even results in a contradiction.    

   \begin{lemma}[Lemma~\ref{RGB3-the:3deg5Representative} in~\cite{Liu2023III}] \label{the:3deg5Representative}
Let $a, b,$ and $c$ be three vertices in a triangle of $EP$ with $\deg(a,b,c)=5$. Only one congruent class of RGB-tilings on $EP-\{ab\}$ exists and has a representative that is shown as $[T_2^{5^3}]$ in Figure~\ref{fig:3deg5a} (but the red dashed line is shown as a solid line).  
   \end{lemma} 

Another method of proving the existence of the red dashed line $K_r|_{v_1}^{v_3}$ in Figure~\ref{fig:3deg5a} is as follows: In Figure~\ref{fig:3deg5a2}, starting with the original $[T_2^{5^3}]$,  we perform ECS twice, first on $rGCL$ and then on $bCL$. The result $[T_2^{5^3}+2_{rb}]$ is given as the third graph. Because the red generalized canal ring crosses $\Sigma'$ and the second blue canal ring is all inside $\Sigma$, the new red Kempe chain $K_r|_{v_1}^{v_3}$ should exist in $[T_2^{5^3}]$ before this modification.    
   \begin{figure}[h]
   \begin{center}
   \includegraphics[scale=0.87]{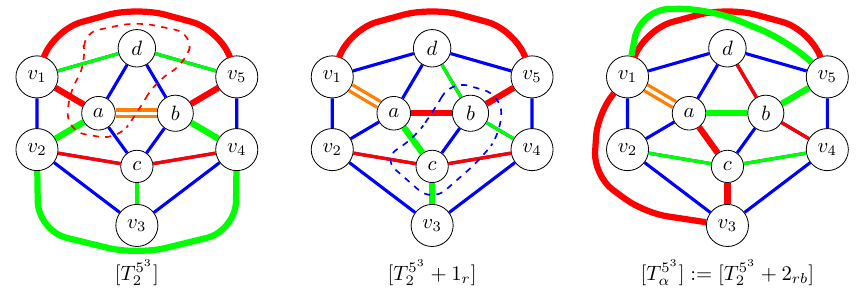}
   \end{center}
\caption{Another method of realizing $K_r|_{v_1}^{v^3}$} \label{fig:3deg5a2}  
   \end{figure} 

This third graph is crucial and is thus denoted  as $[T_\alpha^{5^3}]$.  In this eRGB-tiling, all edges along $\Phi$ are blue and have three degree 5 vertices inside. Although the structure and edge-coloring seem to be symmetric, the graph $[T_\alpha^{5^3}]$ in Figure~\ref{fig:3deg5a2} is not actually symmetric because a green Kempe chain $K_g|_{v_3}^{v_5}$ is missing. Symmetry is not the reason that $K_g|_{v_3}^{v_5}$ exists; the reason can be found in Figure~\ref{fig:3deg5b}. Moreover, drawing $K_r|_{v_1}^{v_3}$ together with $K_r|_{v_1}^{v_5}$ is the same as $K_r|_{v_1}^{v_3}$ together with $K_r|_{v_3}^{v_5}$ or $K_r|_{v_1}^{v_5}$ together with $K_r|_{v_3}^{v_5}$. {\bf This is because the requirement is that $v_1$, $v_3,$ and $v_5$ are both red-connected and also green-connected.}

As an exercise, we develop three congruent RGB-tilings on $EP-\{e\}$ for the possibility of $e\in \{av_1,bv_5,cv_3\}$ in Figure~\ref{fig:3deg5b}. All edges inside $\Sigma$ can be painted in six ways such that only one $e$-diamond has an yellow double-line.    
   \begin{figure}[h]
   \begin{center}
   \includegraphics[scale=0.87]{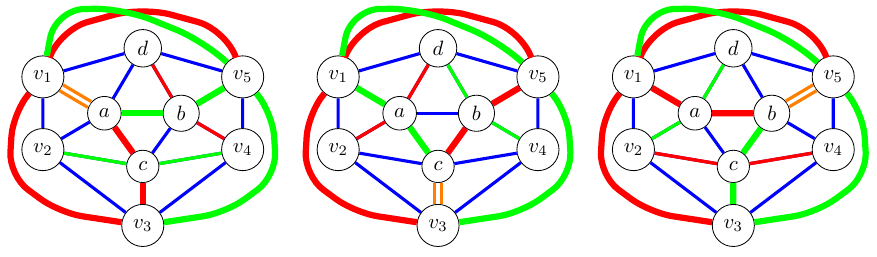}
   \end{center}
   \caption{Three congruent RGB-tilings. All of them are $[T_\alpha^{5^3}]$}     \label{fig:3deg5b} 
   \end{figure}
        \begin{lemma}[Lemma~\ref{RGB3-thm:deg5inT6blue} in~\cite{Liu2023III}] \label{thm:deg5inT6blue}
Let $EP\in e\mathcal{MPGN}4$ with $T\!D:=(\{a,b,c: \text{in a triangle}\};\deg(a,b,c)=5)$. The three eRGB-tilings $[T_\alpha^{5^3}]$ shown in Figure~\ref{fig:3deg5b} are equivalent to each other. Six such equivalent tilings exist inside $\Sigma$; $T\!D$ can be discussed by using these six tilings flexibly.  
        \end{lemma}

\subsection{Four degree 5 vertices in a diamond}   \label{sec:4deg5inDiamond}
This section answers the following interesting question: Can a diamond in $EP$ have all of its four vertices as degree 5?

   \begin{theorem}[Theorem~\ref{RGB3-thm:4deg5inDiamond} in~\cite{Liu2023III}]    \label{thm:4deg5inDiamond}
Let $a, b, c, d$ be four vertices in a diamond in $EP$. Having all of them as degree 5 is impossible.
   \end{theorem}
   \begin{proof}
Let $T\!D:=(\{a,b,c,d: \text{in a diamond}\};\deg(a,b,c,d)=5)$ or $5^4$ for short. Let $\Omega:=v_1$-$v_2$-$\ldots$-$v_6$-$v_1$ be the 6-cycle of the neighbors of $T\!D$. 

Let us adopt the second graph $[T_2^{5^3}]$ in Figure~\ref{fig:3deg5a} to fit
$\{a,b,c\}$ and $\{a,b,d\}$; we then obtain only the initial status $[T^{5^4}],$ as illustrated in Figure~\ref{fig:4deg5}. Before we proceed to the main proof, the condition that $v_3\neq v_6$ must be checked.  Theorem~\ref{thm:nontrivial34}(2) offers a proof. Additionally, we known the red-connectivity of $v_1$, $v_3,$ and $v_5$; these three vertices are also red-disconnected with $v_6$. This is another proof that $v_3\neq v_6$. However, $v_1\neq v_3, v_4, v_5$ and similar cases have not been checked. Theorem~\ref{thm:nontrivial34}(1) and $EP$ having its minimum degree 5 offer proofs for these simple cases, respectively.       
   \begin{figure}[h]
   \begin{center}
   \includegraphics[scale=0.9]{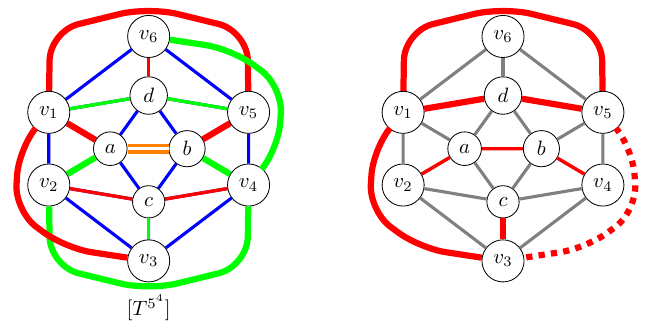}
   \end{center}
   \caption{Diamond with four degree 5's in $EP$}  
   \label{fig:4deg5} 
   \end{figure}

Now we simply rearrange to obtain a new red tiling inside $\Sigma,$ as shown in the second graph in Figure~\ref{fig:4deg5}. The displayed R-tiling has no odd-cycles. If a new cycle is added, it must cross $\Sigma$. First, $K_r|_{v_1}^{v_5}\cup \{dv_1, dv_5\}$ is an even-cycle. The path $v_2$-$a$-$b$-$v_4$ cannot be extended to a larger cycle because it is blocked by $K_r|_{v_1}^{v_3}$. Finally, if $K_r|_{v_5}^{v_3}$ (red dashed line) exists, its length must be even, as revealed by examining the first graph, where $v_5$-$v_4$-$v_3$ has length 2 and all blue. Thus, the largest cycle $K_r|_{v_1}^{v_3}\cup K_r|_{v_5}^{v_3}\cup \{dv_1, dv_5\}$ has even length. According to Theorem~\ref{thm:4RGBtiling}, an R-tiling without red odd-cycle on an MPG must induce a 4-coloring function. This is a contradiction and the proof is done. 
   \end{proof} 

\subsection{Three degree 5 vertices in a triangle, continued}
\label{sec:3degree5cont}

We revisit $T\!D:=(\{a,b,c: \text{in a triangle}\};\deg(a,b,c)=5)$; we must next refer to Lemma~\ref{thm:deg5inT6blue}. Figure~\ref{fig:3deg5b} presents three equivalent RGB-tilings on $\Sigma'$ associated with $T\!D$; that is, a skeleton must exist in $\Sigma'$ if all edges along $\Omega$ are blue. By redrawing that skeleton in $\Sigma'$ but leaving every edge inside $\Sigma$ black, we can obtain the first graph in Figure~\ref{fig:3deg5bb}.  Six vertices, for example, $u_1, u_2, \ldots, u_6$, must surround $T\!D$ because according to Theorem~\ref{thm:4deg5inDiamond}, $\deg(d, v_2, v_4)\ge 6$. {\bf Here, we thus assume the minimum situation $(\ast)$: $\deg(d, v_2, v_4)= 6$ and $\deg(v_1, v_3, v_5)= 5$.} Moreover, assume that $\hat{T\!D}$ has the vertex set $\{a,b,c,d,v_1,\ldots,v_5\}$ and the $(\ast)$ requirement; this results in a new $\hat{\Omega}:=u_1$-$u_2$-$\ldots$-$u_6$ and a new $\hat{\Sigma}$ and $\hat{\Sigma}'$ inside and outside of $\hat{\Omega}$, respectively, as indicated in the second graph in Figure~\ref{fig:3deg5bb}. 

The second graph is the only feasible RGB-tiling on $\Sigma'$ (not only on $\hat{\Sigma}'$) up to synonyms. In particular, all edges in $\hat{\Omega}$ must be blue. Clearly we can find a blue canal ring $bCL$ in between $\Omega$ and $\hat{\Omega}$. After performing ECS on this $bCL$, we obtain a new RGB-tiling on $\Sigma',$ which is shown as the third graph in Figure~\ref{fig:3deg5bb}. Moreover, we can perform $\Sigma$-adjustment by coloring the paths $v_1$-$a$-$b$-$v_5$ and $v_2$-$c$-$v_4$ red. Finally, we obtain a new R-tiling without red odd-cycles. The only red cycle crossing $\hat{\Sigma}$ must have even length.     
   \begin{figure}[h]
   \begin{center}
   \includegraphics[scale=0.9]{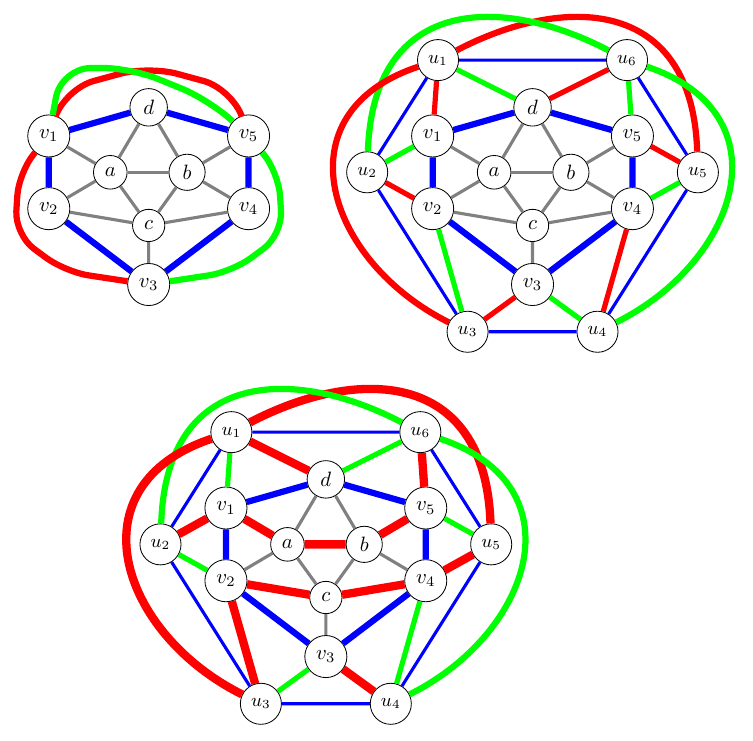}
   \end{center}
   \caption{New $\hat{T\!D}$: a union of $5^3$  and the surrounding $(56)^3$}     \label{fig:3deg5bb} 
   \end{figure}  

        \begin{lemma}[Lemma~\ref{RGB3-thm:5inTri565656} in~\cite{Liu2023III}]  \label{thm:5inTri565656}
Let $EP\in e\mathcal{MPGN}4$ with $a,b,c\in V(EP)$ in a triangle and $\deg(a,b,c)=5,$ as displayed in Figure~\ref{fig:3deg5bb}. The surrounding vertices along $\Omega:=d$-$v_1$-$v_2$-$v_3$-$v_4$-$v_6$ cannot have the following degree property: $\deg(d,v_2,v_4)=6$ and $\deg(v_1,v_3,v_5)=5$.       
        \end{lemma}

	\begin{remark}
Let us focus on the third graph in Figure~\ref{fig:3deg5bb}. We see $dv_1$ and $dv_5$ blue, then how to color $da$ and $db$? Actually, we shall recolor green and blue along $rCL(da)$ which is the red canal ring passing through $da$. Since this R-tiling has no red odd-cycle, the length of $rCL(da)$ is even and recoloring is feasible. The whole idea is summarized by Theorem~\ref{thm:4RGBtiling}. Once we apply this theorem, we can simply treat green and blue as black.    
	\end{remark}

\subsection{Better result than Lemma~\ref{thm:5inTri565656}}   \label{sec:ABetterResult}

The topic of discussion in Lemma~\ref{thm:5inTri565656} pertains to  $\hat{T\!D}$ with nine vertices. Here, we show a stronger property with a new $T\!D \subseteq \hat{T\!D}$; that is, $T\!D$ is a subgraph of $\hat{T\!D}$.     

        \begin{theorem}  \label{thm:5inTri56}
Suppose $EP\in e\mathcal{MPGN}4$ with $T\!D:=(\{a,b,c,d,v_1: \text{see Figure~\ref{fig:55556}}\};$ $\deg(a,b,c,v_1)=5,\deg(d)=6)$. This is impossible. 
        \end{theorem}
   \begin{figure}[h]
   \begin{center}
   \includegraphics[scale=0.8]{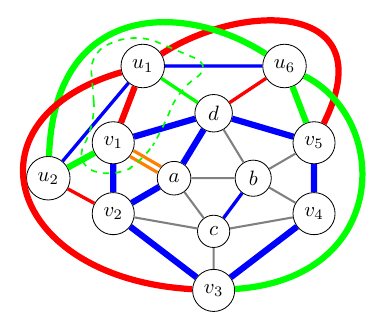}
   \end{center}   
   \caption{New $T\!D$ and its border $\Omega$}     \label{fig:55556} 
   \end{figure}    
   \begin{proof}
For this $T\!D$, the border is $\Omega:=u_1$-$u_2$-$v_2$-$v_3$-$v_4$-$v_5$-$u_6$-$u_1$. Referring to Figure~\ref{fig:55556}, Lemma~\ref{thm:deg5inT6blue} guarantees that the edges of 6-cycle $\Phi:=v_1$-$v_2$-$\ldots$-$v_5$-$d$-$v_1$ must be all blue and we let $ad$ and $av_2$ blue and $av_1$ a yellow double-line.  The coloring pattern of the surrounding edges of blue $cd$ has choices and only use red and green. However, we keep them all black first. Because $\deg(v_1)=5$ and $\deg(d)=6$, the edge-coloring for eight edges outside $\Phi$ but inside $\Sigma$ is unique\footnote{Switching red and green outside $\Phi$ is symmetric; we therefore treat these two results from switching as a single unique result.}. The skeleton in $\Sigma'$ is then inherited from Figure~\ref{thm:deg5inT6blue}.

By performing ECS along the green dashed line in Figure~\ref{fig:55556}, we obtain the first graph in Figure~\ref{fig:55556proof}. This new RGB-tiling is on $EP-\{du_1, v_1u_2\}$; hence, we have double $e$-diamonds. We can now choose an appropriate method of edge coloring inside $\Phi,$ as displayed in Figure~\ref{fig:55556proof}\footnote{This is a key for this proof.}. Assigning $du_1$ and $v_1u_2$ as green (blue) results in a green (blue) odd-cycle passing through $v_1u_2$ ($du_1$ respectively). The new blue Kempe chain passing $du_1$ can be either $K_b|_{u_1}^{v_2}$, $K_b|_{u_1}^{v_3}$, $K_b|_{u_1}^{v_4}$, or $K_b|_{u_1}^{v_5}$.
   \begin{figure}[h]
   \begin{center}
   \includegraphics[scale=0.8]{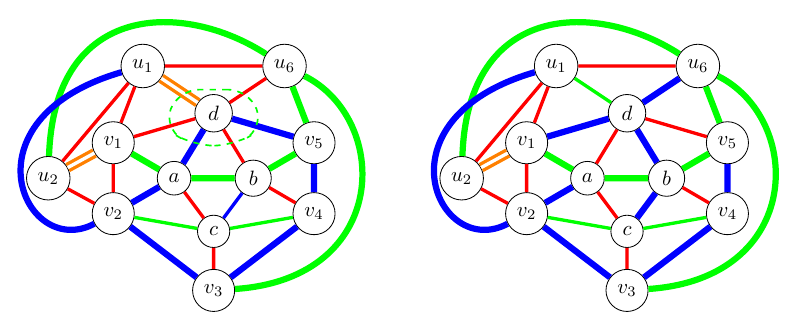}
   \end{center}   
   \caption{Proof of Theorem~\ref{thm:5inTri56}}     \label{fig:55556proof} 
   \end{figure}   

Let us perform another ECS along the green dashed line in the first graph in Figure~\ref{fig:55556proof}. Because this green dashed line stays inside $\Sigma$, this ECS does not change anything in $\Sigma'$\footnote{This is another key for this proof.}. This results in the second graph, which has a single $v_1u_2$-diamond as Type A but no blue odd-cycle; this is a contradiction. 
        \end{proof}

\section{Six vertices of degree 5 in a dumbbell} \label{sec:dumbbell}
Consider a particular $EP\in e\mathcal{MPGN}4$ with six vertices of degree 5 and a dumbbell shape. Figure~\ref{fig:5to6} presents both a dumbbell and another one with six vertices of degree 5. The topic of discussion and the border are respectively
	\begin{eqnarray*}
T\!D &:=& (\{a, b, v_1, v_2, v_4, v_5: \text{in a dumbbell}\};\ \deg(a, b, v_1, v_2, v_4, v_5)=5);\\
\Omega &:=& d\text{-}  u_1\text{-}  u_2\text{-}  u_3\text{-}  c\text{-}  u_4\text{-}  u_5\text{-}  u_6\text{-}  d.
	\end{eqnarray*}
We denote this $T\!D$ as $DumB$.

   \begin{figure}[h]
   \begin{center}
   \includegraphics[scale=0.8]{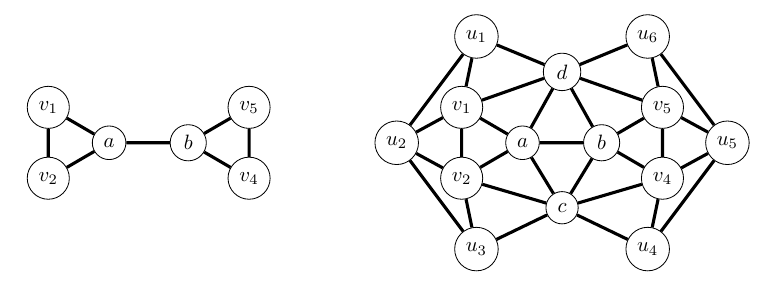}
   \end{center}   
   \caption{Dumbbell $DumB$ and its $\Sigma$}     \label{fig:5to6} 
   \end{figure} 

According to Theorem~\ref{thm:4deg5inDiamond}, we have $\deg(c,d,u_2,u_5)\ge 6$. Furthermore, according to Theorem~\ref{thm:5inTri56}, we have $\deg(c,d)\ge 7$. We plan to prove that for $DumB,$ we have $\deg(c,d)\ge 8$. 

The only two possible RGB-tilings with $ab$ as the $e$-diamond (the single yellow double-line) are denoted as $[T_\alpha^{55}]$ and $[T_\beta^{55}]$ and depicted in Figure~\ref{fig:KempeCRo2deg5basic}; here, $55$ indicates $\deg(a,b)=5$. Each of $[T_\alpha^{55}]$ and $[T_\beta^{55}]$ brings up 5 cases which are denoted by $[X_\alpha]$ and $[X_\beta]$ where $X\in \{A, B, C1, C2, C3\}$. However, $[T_\alpha^{55}] \cong [T_\beta^{55}],$ according to Eq~\ref{eq:T55S09}; hence, we can select either, and we choose $[T_\alpha^{55}]$ for analysis. Please, refer to Figure~\ref{fig:5to6Fivecases}.  In this section, we write $[X_\alpha]$ as $[X]$.

   \begin{figure}[h]
   \begin{center}
   \includegraphics[scale=0.62]
{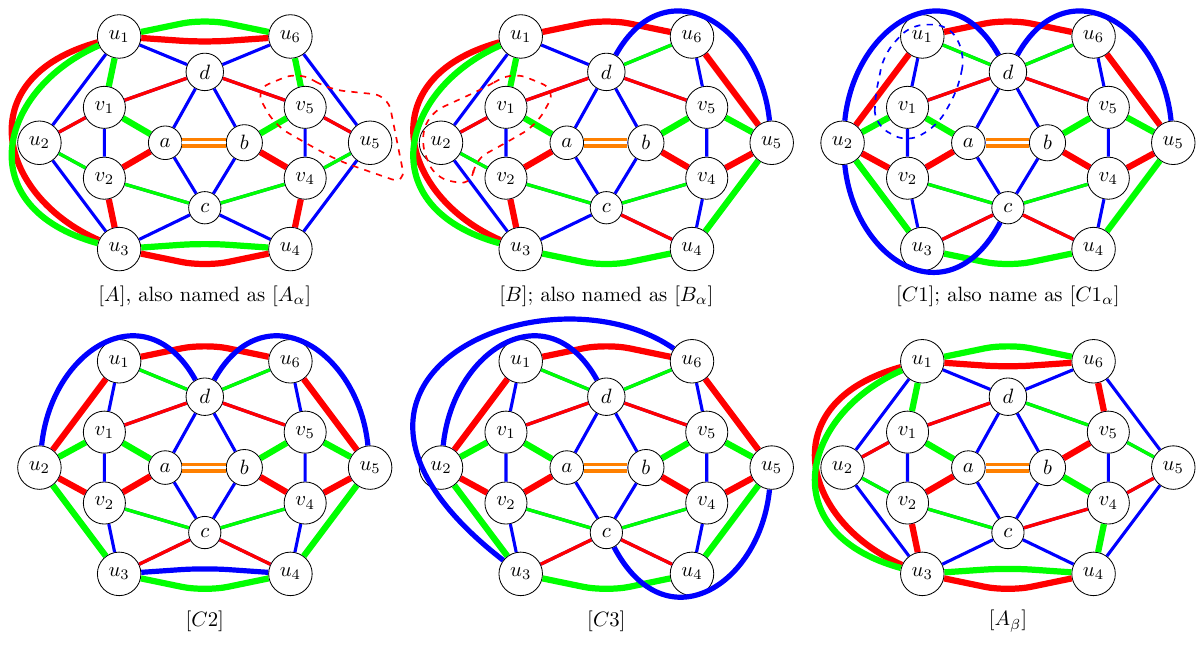}
   \end{center}   
   \caption{Five cases for $T\!D:=DumB$ and $[A_\beta]$}     \label{fig:5to6Fivecases} 
   \end{figure} 

The edges $av_1$, $av_2$, $bv_4,$ and $bv_5$ can be considered four $T\!D$'s, i.e., they are four different $55$; however, their $e$-diamonds are always at edge $ab$. According to $ATLAS$ in Tables~\ref{tb:summary} and~\ref{tb:summary2}, we must only be concerned with $[S_1]$ and $[S_4]$. Put it another way, considering red and green on the triangles $\triangle v_1v_2u_2$ and $\triangle v_4v_5u_5,$ where edges $v_1v_2$ and $v_4v_5$ must be blue, we must handle five cases, as depicted in Figure~\ref{fig:5to6Fivecases}. Listing the insides for these five is trivial, but analyzing the skeleton outside of $\Omega$ requires referring to  $[S_1]$ and $[S_4]$ in Tables~\ref{tb:summary} and~\ref{tb:summary2}. For instance, there are two $[T_\alpha^{5^3}]$ in $[A]$. In $[B]$, $bv_5$ is $55$ of pattern $[S_4]$; therefore, we need to draw a Kempe chain either $K_b|_d^{u_5}$ or $K_b|_c^{u_5}$.   

   \begin{figure}[h]
   \begin{center}
   \includegraphics[scale=0.62]
{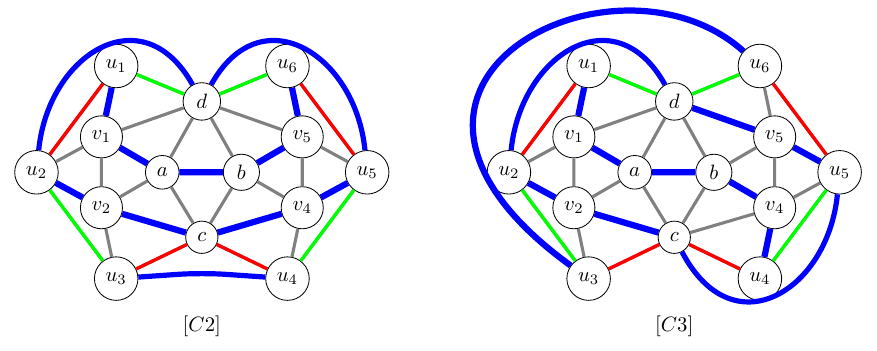}
   \end{center}   
   \caption{Proof by $\Sigma$-adjustment to rule out $[C2]$ and $[C3]$}     \label{fig:CasesC2C3} 
   \end{figure}     
        \begin{lemma} \label{thm:CasesC2C3}
With $T\!D:=DumB$ as defined, refer to Figure~\ref{fig:5to6Fivecases}. Cases $[C2]$ and $[C3]$ are impossible for $EP\in e\mathcal{MPGN}4$.
        \end{lemma}
        \begin{proof}
This can be simply proven by observing the following two graphs given in Figure~\ref{fig:CasesC2C3}. We can easily check that the two B-tilings have no blue odd-cycles; they can thus induce 4-coloring functions. This modification of the edge-coloring inside $\Sigma$ of $[C2]$ and $[C3]$ is called $\Sigma$-adjustment.    
        \end{proof}

Without interference from cases $[C2]$ and $[C3]$, we then focus on a good property about $[A]$, $[B]$ and $[C1]$ as follows.

        \begin{theorem}   \label{thm:CasesABC1}
With $T\!D:=DumB$ as defined, refer to Figure~\ref{fig:5to6Fivecases}. $[A]\cong [B]\cong [C1]$.
        \end{theorem}      
        \begin{proof}
$[A]\cong [B]$: Start with $[A]$ in Figure~\ref{fig:5to6Fivecases} and perform ECS along the red-dashed line (a red canal ring) to obtain the first graph in Figure~\ref{fig:ABC1a}. With double $e$-diamonds, no green Kempe chain involving $ab$ can escape; therefore, a green Kempe chain must pass through the yellow double-line $dv_5$. Therefore, we draw $K_g|_d^{u_2}$ (dashed) and $K_g|_d^c$; at least one of these must exist, and both can exist. However, the dashed $K_g|_d^{u_2}$ should not exist because of the second graph $[A+1']$. In $[A+1']$, $a$-$c$-$v_4$-$v_5$-$d$-$a$ is an odd-cycle of length 5 passing through the $e$-diamond. 
   \begin{figure}[h]
   \begin{center}
   \includegraphics[scale=0.62]
{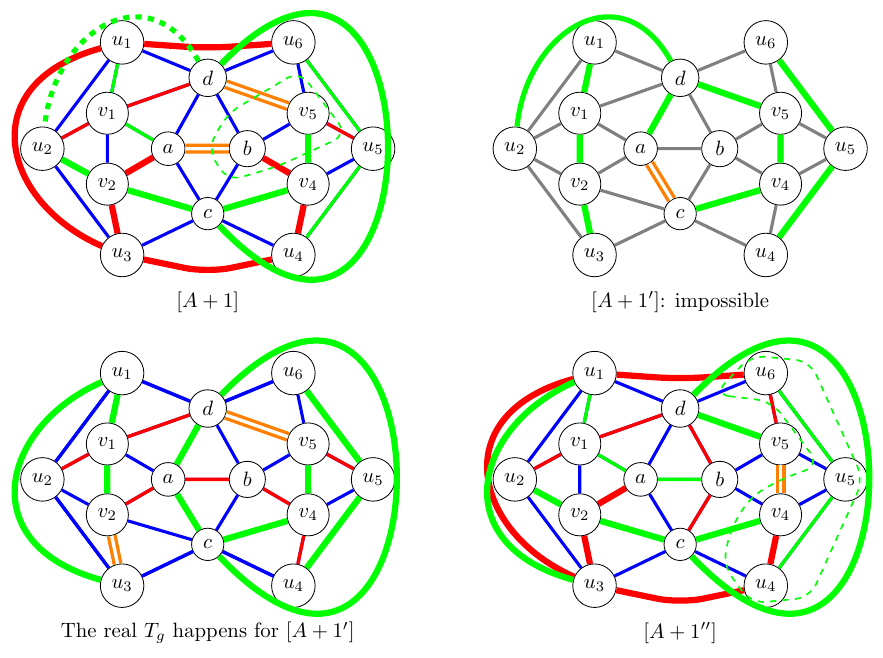}
   \end{center}   
   \caption{Argument on $[A+1]$}     \label{fig:ABC1a} 
   \end{figure} 
This fact contradicts Lemma~\ref{thm:CoalphaNew}(b). To avoid this impossible situation, the real situation must be as in the third graph in Figure~\ref{fig:ABC1a}; that is, it must have two $e$-diamonds, and the dashed $K_g|_d^{u_2}$ shall be replaced by $K_g|_{u_1}^{u_3}$. Then, $K_g|_d^c$ must exist. This conclusion is demonstrated by $[A+1'']$. Note the small change inside $\Sigma$ by observing the thin green dashed ring surrounding $bv_5$ in the first graph $[A+1]$. 

   \begin{figure}[h]
   \begin{center}
   \includegraphics[scale=0.62]
{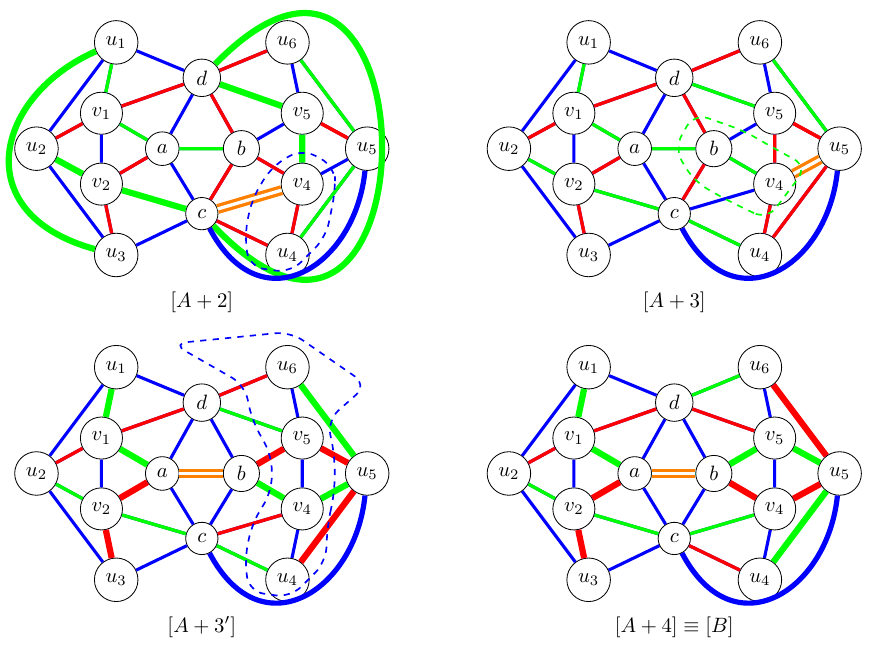}
   \end{center}   
   \caption{$[A+4]\equiv [B]$}     \label{fig:ABC1b}
   \end{figure} 
Starting from $[A+1'']$, we can perform four more ECS operations and then complete the first part of the proof by showing $[A]\cong [A+4]\equiv [B],$ as displayed in Figure~\ref{fig:ABC1b}.

   \begin{figure}[h]
   \begin{center}
   \includegraphics[scale=0.62]
{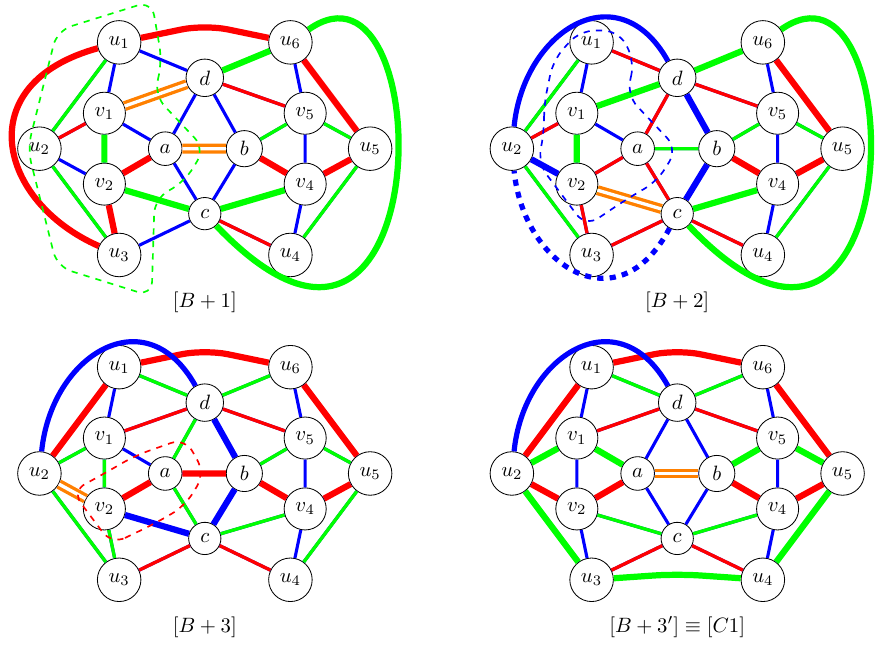}
   \end{center}   
   \caption{$[B+3']\equiv [C1]$}     \label{fig:ABC1c}
   \end{figure} 
$[B]\cong [C1]$: Starting from $[B]$ with a thin red dashed-ring in Figure~\ref{fig:5to6Fivecases}, we obtain the first graph in Figure~\ref{fig:ABC1c}. Keep going to perform three more ECS, then we finish the second part of proof by showing $[B]\cong [C1]$.

   \begin{figure}[h]
   \begin{center}
   \includegraphics[scale=0.62]
{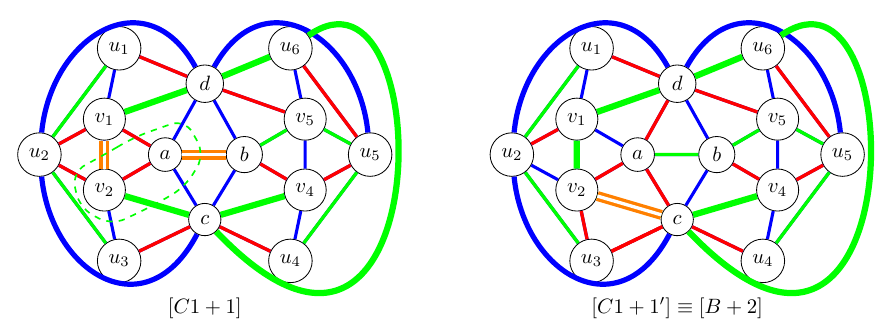}
   \end{center}   
   \caption{$[C1+1']\equiv [B+2]$}     \label{fig:ABC1d}
   \end{figure}    
An important argument for $[B+2]$ in Figure~\ref{fig:ABC1c} need to be claimed clearly. Due to the red $e$-diamond in $[B+2]$, we have at least a blue Kempe chain which is $K_b|_d^{u_2}$ or $K_b|_c^{u_2}$. The thin blue dash-ring $bGCL$ is required not crossing $b$-$c$-$u_4$-$u_5$-$u_6$-$d$-$b$ to avoid any change of them. The way we draw  $bGCL$ is assisted by $K_b|_d^{u_2}$. However, is it possible that only $K_b|_c^{u_2}$ exists but $K_b|_d^{u_2}$ does not? The answer is no. Let us back to $[C1]$ in Figure~\ref{fig:5to6Fivecases}. We perform ECS along that blue dash-ring and obtain $[C1+1]$ in Figure~\ref{fig:ABC1d}. By one more ECS along $gGCL$ which is inside $\Sigma$, we reach $[C1+1']\equiv [B+2]$. Now we confirm that both $K_b|_d^{u_2}$ and $K_b|_c^{u_2}$ exist in $[B+2]$. The equivalence $[C1+1']\equiv [B+2]$ offers another proof for $[B]\cong [C1]$.
	\end{proof}   

	\begin{theorem}   \label{thm:Adeg8}
Follow the setting for $T\!D:=DumB$ and refer to Figure~\ref{fig:5to6Fivecases}. We have $\deg(c,d)\ge 8$. 
	\end{theorem}	
	\begin{proof}
We see $[C1+1']$ having $\deg(c)\ge 8$ in Figure~\ref{fig:ABC1d}. However, we would like to show $\deg(c,d)\ge 8$ by a single graph.

Let us start with $[A+1]$ in Figure~\ref{fig:ABC1a}. It must have $K_g|_{u_1}^{u_3}$. We copy $[A+1]$ as the first graph in Figure~\ref{fig:Adeg8} and perform ECS along the thin green dashed-ring. The result is $[A {+\text{II}}]$. There are three $e$-diamonds in  $[A {+\text{II}}]$. When we replace three yellow double-lines by red and also treat all green and blue edges as black color, we find that at least one of $K_r|_c^d$ and $K_r|_{u_5}^d$ exists.  
   \begin{figure}[h]
   \begin{center}
   \includegraphics[scale=0.62]
{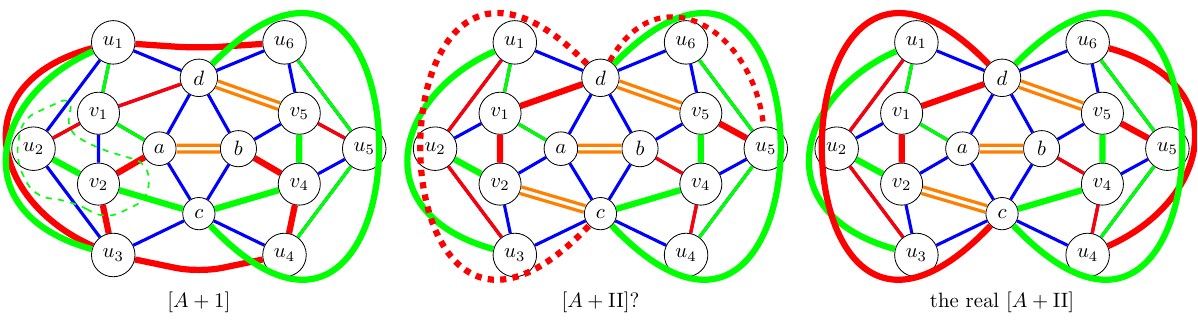}
   \end{center}   
   \caption{Why $\deg(c,d)\ge 8$?}     \label{fig:Adeg8}
   \end{figure} 
We just follow the same discussion about ``$[A+1']$: impossible'' demonstrated in Figure~\ref{fig:ABC1a}, then we can rule $K_r|_{u_5}^d$ out. Therefore, we are sure that $K_r|_c^d$ makes the only red odd-cycle passing through $cv_2$. The existence of $K_r|_{u_4}^{u_6}$ is also a consequence of no $K_r|_{u_5}^d$. 

By $[A{+\text{II}}]$, we know $\deg(c,d)\ge 8$.  
	\end{proof}

	\begin{corollary}  \label{thm:cong}
Given $EP:=(EP;Dumb)\in e\mathcal{MPGN}$, we have a congruent class of eRGB-tilings on $EP$ containing $[X_\alpha]$ and $[X_\beta]$ for for $X\in\{A, B, C1\}$, as well as $[A+1]$, $[A+2]$, $[A+3]$, $[B+1]$, $[B+2]$, $[B+3]$, $[C1+1]$, and $[A+\rm{II}]$. 	
	\end{corollary}

\section{Does $EP^\perp$ belong to $e\mathcal{MPGN}4$}  \label{sec:EPperp}

Let us redraw the original $EP:=(EP; DumB)$ in a new shape as the left graph in Figure~\ref{fig:DumBandM}.  We also draw $EP^\perp$ as the middle graph which is made by the same $\Omega$ and $\Sigma'$ of $EP$, but the new interior $\Sigma^\perp:= \Sigma(EP^\perp; \Omega)$. 
   \begin{figure}[h]
   \begin{center}
   \includegraphics[scale=0.64]{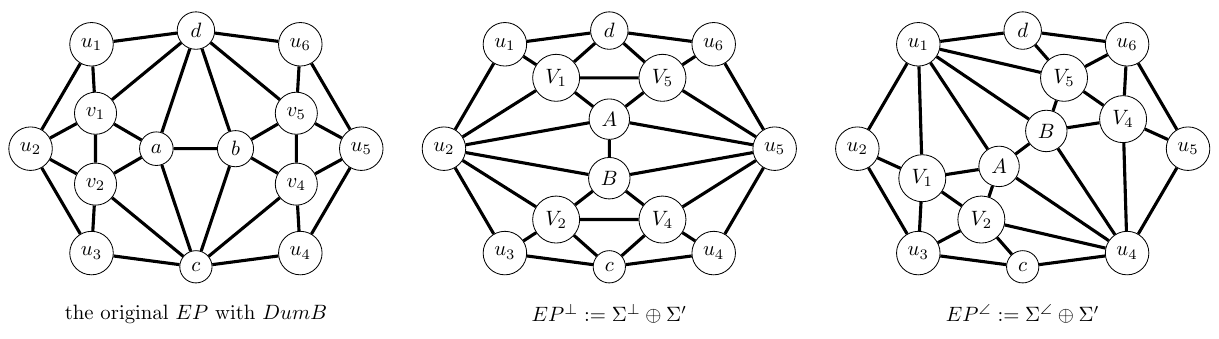}
   \end{center}   
   \caption{The original $EP:=\Sigma\oplus \Sigma'$ and $\Sigma^\perp\oplus \Sigma'$, $\Sigma^\angle  \oplus \Sigma'$}     \label{fig:DumBandM} 
   \end{figure} 
The topic of discussion of $EP^\perp$ is $DumB^\perp =\{A, B, V_1, V_2, V_4, V_5\}$ and this $\Sigma^\perp$ is obtained by rotating the original $DumB$ with $90^\circ$ and still keeping $\deg(A, B, V_1, V_2, V_4, V_5)=5$. 

The new MPG $EP^\angle:=\Sigma^\angle  \oplus \Sigma'$ shown as the right graph in Figure~\ref{fig:DumBandM} is defined in a similar way. We are interested to determine whether $EP^\perp$ and $EP^\angle$ belong to $e\mathcal{MPGN}4$ or not. 

Besides $EP^\perp$, a similar question is regarding $(EP;5^3)^\triangle$. Here the new MPG $(EP;5^3)^\triangle$, shown as the right graph in Figure~\ref{fig:555}, has $T\!D=(\{A,B,C\};\deg(A,B,C)=5)$ and the same $\Sigma'$ from the original $(EP;5^3)$ given in Figures~\ref{fig:3deg5a2} and~\ref{fig:3deg5b}. 
   \begin{figure}[h]
   \begin{center}
   \includegraphics[scale=0.75]{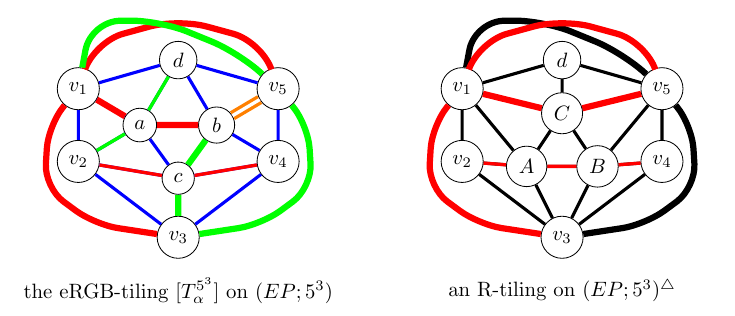}
   \end{center}   
   \caption{The original $(EP;5^3)$ and $(EP;5^3)^\triangle$}     \label{fig:555} 
	\end{figure} 

	\begin{theorem}  \label{thm:DumBandM4col}
(a)	If $(EP;5^3) \in e\mathcal{MPGN}4$  then $(EP;5^3)^\triangle$ is 4-colorable. (b) If $(EP; DumB) \in e\mathcal{MPGN}4$  then both $EP^\perp$ and $EP^\angle$ are 4-colorable.
	\end{theorem}
	\begin{proof}
(a): The right graph in Figure~\ref{fig:555} shows an R-labeling on $(EP;5^3)^\triangle$ that uses the portion on $\Sigma'$ obtained from $[T^{5^3}_\alpha]$. Obviously this R-labeling has no red odd-cycles. Thus, $(EP;5^3)^\triangle$ is 4-colorable by Theorem~\ref{thm:4RGBtiling}.

(b): By Corollary~\ref{thm:cong}, the eRGB-tiling $[A+\rm{II}]$ on $(EP; DumB)$ exists. 
The left graph in Figure~\ref{fig:DumBandM007} shows an B-labeling on $(EP; DumB^\perp)$ that uses the portion on $\Sigma'$ obtained from $[A+\rm{II}]$ in Figure~\ref{fig:Adeg8}.	Since we are considering a B-labeling, all red and green edges are now colored gray, except for two red and two green edges along $\Omega$. Inside $\Sigma^\perp$, we also have a B-tiling, particularly a blue hexagon. 
    \begin{figure}[h]
   \begin{center}
   \includegraphics[scale=0.7]{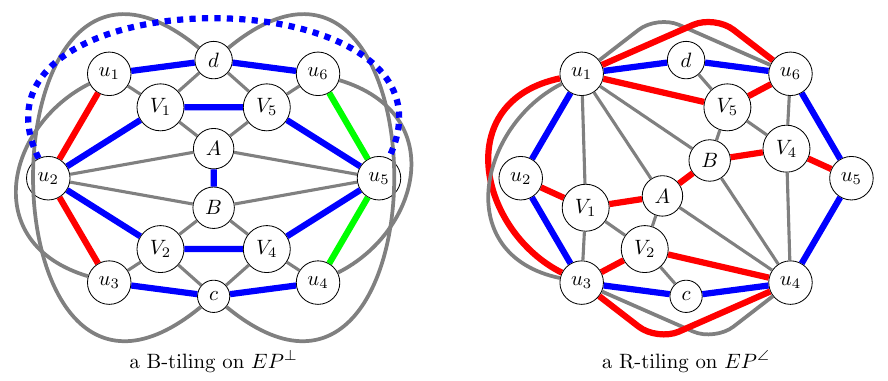}
   \end{center}   
   \caption{$EP^\perp$ and $EP^\angle$ are 4-colorable}     \label{fig:DumBandM007} 
   \end{figure} 
A possible Kempe chain $K_b|_{u_2}^{u_5}$ is shown as a dashed line. Considering the cycle $K_b|_{u_2}^{u_5}\cup$ $u_2$-$u_1$-$d$-$u_6$-$u_5$ and  Lemma~\ref{thm:evenoddRGB}(b), we can conclude that $K_b|_{u_2}^{u_5}$ must be of odd length. Therefore, the graph depicts a B-tiling on $EP^\perp$ without any blue odd-cycles, even if the blue dashed $K_b|_{u_2}^{u_5}$ exists. Thus, $EP^\perp$ is 4-colorable by Theorem~\ref{thm:4RGBtiling}.  

As for $EP^\angle$, the proof is easy. We use the portion of R-tiling on $\Sigma'$ obtained from $[A_\alpha]$ and $[A_\beta]$, which are actually the same, in Figure~\ref{fig:5to6Fivecases}. The key property is that $v_1$, $v_3$, $v_4$, and $v_6$ are red-connected in $\Sigma'$. The right graph in Figure~\ref{fig:DumBandM007} depicts an R-tiling on $EP^\angle$ without any red odd-cycles but two red even-cycles crossing $\Sigma$. Then the proof is complete. 
   	\end{proof}

There are three similar results given as Theorem~\ref{RGB2-thm:EPNScolorable}(b) in~\cite{Liu2023II}, Theorem~\ref{RGB3-thm:allGraphsSimilarToEP} and Remark~\ref{RGB3-re:MaPlus} in~\cite{Liu2023III}. These results follow a specific pattern: Given a particular $(EP;T\!D)\in e\mathcal{MPGN}4$, we consider a new MPG $(EP;T\!D^{\rm{new}})$.  The new MPG includes the same $\Sigma'$ as $(EP;T\!D)$ and a new $T\!D^{\rm{new}}$ with either $|T\!D^{\rm{new}}|=|T\!D|$ or $|T\!D|+1$. Once the author conjectured that $EP^\perp$ belongs to $e\mathcal{MPGN}4$. However, this conjecture turned out to be false. So far, the author has not found any $T\!D^{\rm{new}}$ that makes $(EP;T\!D^{\rm{new}})$ non-4-colorable.  Searching for non-4-colorable $(EP;T\!D^{\rm{new}})$'s might be an interesting subject.

\bibliographystyle{amsplain}

\begin{thebibliography}{10}

\bibitem{Kempe1879} A.B.\ Kempe , \textit{On the Geographical Problem of the Four Colours}, Am.\ J.\ Math.\ \textbf{2} (1879), 193--220.


\bibitem{Liu2023I} S.-C. Liu,  \textit{A renewal approach to prove the Four Color Theorem unplugged, Part I: RGB-tilings on maximal planar graphs}, preprint, https://arxiv.org/abs/2309.11733.

\bibitem{Liu2023II} S.-C. Liu,  \textit{A renewal approach to prove the Four Color Theorem unplugged, Part II: R/G/B Kempe chains in an extremum non-4-colorable MPG}, preprint, https://arxiv.org/abs/2309.09999.

\bibitem{Liu2023III} S.-C. Liu,  \textit{A renewal approach to prove the Four Color Theorem unplugged, Part III: Diamond routes, canal lines and $Sigma$-adjustments}, preprint, https://arxiv.org/abs/2309.09998.




\end{thebibliography}

\end{document}